\documentclass[11pt]{amsart}
\usepackage{amsmath,amssymb,amsfonts,url,mathptmx}
\usepackage{color}
\numberwithin{equation}{section}
\newtheorem{Def}{Definition}[section]
\newtheorem{thm}{Theorem}[section]
\newtheorem{lem}{Lemma}[section]
\newtheorem{rem}{Remark}[section]
\newtheorem{prop}{Proposition}[section]

\newcommand{\hdot}{^\text{\r{}}\hspace{-.33cm}H}
\begin{document}
\title[Toda system]{On Rank Two Toda System with Arbitrary Singularities: Local Mass and New Estimates} \subjclass{35J60, 35J55}
\keywords{SU(n+1)-Toda system, asymptotic analysis, a priori
estimate, classification theorem, topological degree, blowup
solutions, Riemann-Hurwitz Theorem}

\author{Chang-shou Lin}
\address{Department of Mathematics\\
        Taida Institute of Mathematical Sciences\\
        National Taiwan University\\
         Taipei 106, Taiwan }
\email{cslin@math.ntu.edu.tw}

\author{Juncheng Wei}
\address{Department of Mathematics \\ University of British Columbia \\
 Vancouver, B.C., Canada V6T 1Z2}
\email{jcwei@math.ubc.ca}

\author{WEN YANG}
\address{Center for Advanced Study in Theoretical Sciences (CASTS)\\
        National Taiwan University\\
         Taipei 106, Taiwan }
\email{math.yangwen@gmail.com}

\author{Lei Zhang}
\address{Department of Mathematics\\
        University of Florida\\
        358 Little Hall P.O.Box 118105\\
        Gainesville FL 32611-8105}
\email{leizhang@ufl.edu}

\date{\today}

\begin{abstract} For all rank two Toda systems with an arbitrary singular source, we use a unified approach to prove: (i) The pair of local masses $(\sigma_1,\sigma_2)$ at each blowup point has the expression $$\sigma_i=2(N_{i1}\mu_1+N_{i2}\mu_2+N_{i3}),$$ where $N_{ij}\in\mathbb{Z},~i=1,2,~j=1,2,3.$ (ii) Suppose at each vortex point $p_t$, $(\alpha_1^t,\alpha_2^t)$ are integers and $\rho_i\notin 4\pi\mathbb{N}$, then all the solutions of Toda systems are uniformly bounded. (iii) If the blow up point $q$ is not a vortex point, then $$u^k(x)+2\log|x-x^k|\leq C,$$ where $x^k$ is the local maximum point of $u^k$ near $q$. (iv) If the blow up point $q$ is a vortex point $p_t$ and $\alpha_t^1,\alpha_t^2$ and $1$ are linearly independent over $Q$, then $$u^k(x)+2\log|x-p_t|\leq C.$$ The Harnack type inequalities of (iii) or (iv) are important for studying the behavior of bubbling solutions near each blowup points.
\end{abstract}


\maketitle

\section{Introduction}

Let $(M,g)$ be a Riemann surface without boundary and $\mathbf{K}=(k_{ij})_{n\times n}$ be the Cartan matrix of a simple Lie algebra of rank $n$. For example for the Lie algebra $sl(n+1)$ (the so called $A_n$) we have
\begin{equation}\label{cartan-m}
\mathbf{K}=\left(\begin{array}{ccccc}
2 & -1 & 0 & ... & 0 \\
-1 & 2 & -1 & ... & 0 \\
\vdots & \vdots &  &  & \vdots\\
0 & ... & -1 & 2 & -1 \\
0 & ... &  0 & -1 & 2
\end{array}
\right ).
\end{equation}
In this paper we consider solution  $u=(u_1,...,u_n)$ of the following system defined on $M$:
\begin{equation}\label{e-gen}
\Delta_g u_i+\sum_{j=1}^n k_{ij}\rho_j(\frac{h_je^{u_j}}{\int_M h_j e^{u_j}dV_g}-1)=\sum_{P_{t} \in S} 4\pi \alpha_t^i (\delta_{P_t}-1),
\end{equation}
where $\Delta_g$ is the Laplace-Beltrami operator ($-\Delta_g\ge 0$), $h_1,...,h_n$ are positive and smooth functions on $M$, $\alpha_t^i>-1$ is the strength of the Dirac mass $\delta_{P_t}$, $\rho=(\rho_1,...,\rho_n)$ is a constant vector with nonnegative components. Here for simplicity we
just assume that the total area of $M$ is $1$.

Obviously, equation (\ref{e-gen}) remains the same if $u_i$ is replaced by $u_i+c_i$ for any constant $c_i$. Thus we might assume that each component of $u=(u_1,...,u_n)$ is in
$$ \hdot^{1}(M):=\{v\in L^2(M);\quad \nabla v\in L^2(M), \mbox{and }\,\, \int_M v dV_g=0\}. $$
Then equation (\ref{e-gen}) is the Euler-Lagrange equation for the following nonlinear functional $J_{\rho}(u)$ in $\mathring{H}^{1}(M)$:
$$J_{\rho}(u)=\frac 12\int_M\sum_{i,j=1}^nk^{ij}\nabla_gu_i\nabla_gu_jdV_g-{\sum_{i=1}^n\rho_i\log \int_M h_ie^{u_i}dV_g,}$$
where $(k^{ij})_{n\times n}=\mathbf{K}^{-1}$.

It is hard to overestimate the importance of system (\ref{e-gen}), as it covers a large number of equations and systems deeply rooted in geometry and Physics. Even if (\ref{e-gen}) is reduced to a single equation with Dirac sources, it is a mean field equation that has been extensively studied for decades. The singular sources on the right hand side of the mean field equation describe conic singularities and solutions can be interpreted as metrics with prescribed conic singularities.  This is a classical problem in differential geometry and extensive references can be found in  \cite{bart2,malchiodi-b,erem-3,lin-wei-ye,lwz-apde,tr-1,tr-2,yang1} etc.  Recently profound relations among mean field equation, classical Lame equation, hyper-elliptic curves, modular forms and
Painleve equation have been discovered and developed (see \cite{chai} and \cite{CKL}).

When (\ref{e-gen}) has more than one equation, it has close ties with algebraic geometry and integrable system. For example, solutions of the $sl$($n+1$) Toda system are closely related to holomorphic curves in projective spaces.
Let $f$ be a holomorphic curve from a domain $D$ of $\mathbb{R}^{2}$ into $\mathbb C\mathbb P^n$. Lift locally $f$ to $\mathbb C^{n+1}$ and denote the lift by $\nu(z)=[\nu_0(z),...,\nu_n(z)]$. The $k$th associated curve of $f$ is defined by
$$f_k:D\to G(k,n+1)\subset \mathbb C\mathbb P^n(\Lambda^k \mathbb C^{n+1}),\quad f_k(z)=[\nu(z)\wedge\nu'(z)\wedge...\wedge \nu^{(k-1)}(z)], $$
where $\nu^{(j)}$ is the $j-$th derivative of $\nu$ with respect to $z$. Let
$$\Lambda_k(z)=\nu(z)\wedge ...\wedge \nu^{(k-1)}(z), $$
then the well known infinitesimal Pl\"uker formula gives
\begin{equation}
\label{plucker}
\frac{\partial^2}{\partial z\partial \bar z}\log \|\Lambda_k(z)\|^2=\frac{\|\Lambda_{k-1}(z)\|^2\|\Lambda_{k+1}(z)\|^2}{\|\Lambda_k(z)\|^4}, \quad \mbox{ for } k=1,2,..,n,
\end{equation}
where we define the norm $\|\cdot \|^2=\langle\cdot,\cdot\rangle$ by the Fubini-Study metric in $\mathbb C\mathbb P(\Lambda^k \mathbb C^{n+1})$ and put $\|\Lambda_0(z)\|^2=1$. We observe that (\ref{plucker})
holds only for $\|\Lambda_k(z)\|>0$, i.e. all the unramification points $z\in M$. Setting $\|\Lambda_{n+1}(z)\|=1$ by nomalization (analytically extended at the ramification points) and
$$U_k(z)=-\log \|\Lambda_k(z)\|^2+k(n-k+1)\log 2, \quad 1\le k\le n. $$
Let $p$ be a ramified point and $\{ \gamma_{p,1}, \cdots ,\gamma_{p,n} \}$ be the total ramification index at $p$. Write:
 $$u_i^*=\sum_{j=1}^nk_{ij}U_j, \quad \alpha_{p,i}=\sum_{j=1}^n k_{ij}\gamma_{p,j},$$
then we have
\begin{equation}\label{e-uis}
\Delta u_i^*+\sum_{j=1}^n k_{ij} e^{u_j^*}-K_0=4\pi \sum_{p\in S}\alpha_{p,i}\delta_p, \quad i=1,...,n,
\end{equation}
where $K_{0}$ is the Gaussian curvature of the metric $g$.


Therefore any holomorphic curve from $M$ to $\mathbb C\mathbb P^n$ is associated with a solution $u^*=(u^*_1,...,u^*_n)$ of (\ref{e-uis}). Conversely, given any solution $u^*=(u^*_1,...,u^*_n)$ of (\ref{e-uis}) in $\mathbb S^2$, we can construct a holomorphic curve of $\mathbb S^2$ into $\mathbb C\mathbb P^n$, which has the given ramification index $\gamma_{p,i}$
at $p$. One can see \cite{lin-wei-ye} for the detail of this construction. Therefore, equation (\ref{e-uis}) is related to the following problem: Given a set of ramificated points and its ramification indexes at these points, can we find holomorphic curves into $\mathbb C\mathbb P^n$ that satisfy the given ramification information?


On the other hand, equation (\ref{e-gen}) is also related to many physical models from gauge field theory. For example, to describe the physics of high critical temperature superconductivity, a model of relative Chern-Simons model was proposed and this model can be reduced to a $n\times n$ system with exponential nonlinearity if the gauge potential and the Higgs field are algebraically restricted. Then the Toda system with (\ref{cartan-m}) is one of the limiting equations if the coupling constant tends to zero. For extensive discussions on the relationship between Toda system and its background in Physics we refer the readers to \cite{bennet,ganoulis,lee2,mansfield,yang2} and the reference therein.

In this article we are concerned with rank $2$ Today systems.  There are three types of Cartan matrices of rank 2:
$$A_2=\left(\begin{array}{cc}
2 & -1 \\
-1 & 2
\end{array}
\right )
\quad B_2(=C_2)=\left(\begin{array}{cc}
 2 & -1 \\
 -2 & 2
 \end{array}
 \right )
 \quad G_2=\left(\begin{array}{cc}
 2 & -1 \\
 -3 & 2
 \end{array}
 \right ).
 $$

One of our main theorems is the following estimate:
\begin{thm}
\label{main-manifold}
Let $(k_{ij})_{2\times 2}$ be one of the matrices above, $h_i$ be positive $C^1$ functions on $M$, $\alpha_t^i\in \mathbb N\cup \{0\}$, $t\in \{1,2,...,N\}$ and $K$ be a compact subset of $M\setminus S$. If $\rho_i\not \in 4\pi \mathbb N$, there exists a constant $C(K,\rho_1,\rho_2)$ such that for any solution $u=(u_1,u_2)$ of (\ref{e-gen})
$$|u_i(x)|\le C, \quad \forall x\in K, \quad i=1,2. $$
\end{thm}

Our proof of Theorem \ref{main-manifold} is based on the analysis of the behavior of solutions $u^k=(u_1^k,u_2^k)$ near each blowup point. A point $p\in M$ is called a blowup point if we write
$\tilde u_i^k(x)=u_i^k(x)+4\pi \sum_t\alpha_t^kG(x,p_t)$, where $G(x,y)$ is the Green's function of the Laplacian operator on $M$ with singularities at $y\in M$, and there exists a sequence of points $p_k\to p$ such that $\max_{i=1,2}\{u_1^k(p_k),u_2^k(p_k)\}\to \infty$.


Suppose $u^k$ is a sequence of solutions of (\ref{e-gen}). When $n=1$, it has been proved that if $u^k$ blows up somewhere, the mass distribution $\frac{\rho he^{u^k}}{\int_M he^{u^k}}$ will concentrate, that is,
\begin{equation*}
\label{d-concen}
\frac{\rho he^{u^k}}{\int_M he^{u^k}}\to \sum_{i=1}^S m_i \delta_{p_i}, \quad \mbox{ as } k \to \infty,
\end{equation*}
Which is equivalent to the fact that $u^k(x)\to -\infty$ if $x$ is not a blowup point. This ``blowup implies concentration" was first noted by Brezis-Merle \cite{bm} and was later proved by Li \cite{licmp}, Li-Shafrir \cite{li-shafrir} and Bartolucci-Tarantello \cite{bart2}. But for $n\ge 2$, this phenomenon might fail in general. A component $u_i^k$ is called not concentrating if $u_i^k\not \to -\infty$ away from blowup points, or equivalently, $\tilde u_i^k$ converges to some smooth function $w_i$ away from blowup points. It is natural to ask whether it is possible to have all components not concentrating. For $n=2$, we prove it is impossible.

\begin{thm}\label{one-concen}
Suppose $u^k$ is a sequence of blowup solutions of a rank $2$ Toda system (\ref{e-gen}). Then at least one component of $u^k$ satisfies $u_i^k(x)\to -\infty$ if $x$ is not contained in the blowup set.
\end{thm}

The first example of such non-concentration phenomenon was first proved by Lin-Tarantello \cite{lin-tarantello}. The new phenomenon makes the study of systems $(n\ge 2)$ much more difficult than the mean
field equation ($n=1$). We conjecture that Theorem \ref{one-concen} also holds for $n\ge 3$. This will be studied in a forthcoming project.

As mentioned before our proofs of Theorem \ref{main-manifold} and Theorem \ref{one-concen} are based on the asymptotic behavior of local bubbling solutions. For simplicity we set up the situation as follows:

Let $u^k=(u_1^k,u_2^k)$ be a sequence of solutions of
\begin{equation}\label{toda-b2g2}
\Delta u_i^k +\sum_{j=1}^2 k_{ij} h_j^k e^{u_j^k}=4\pi \alpha_i\delta_0, \quad \mbox{ in }\quad B(0,1),\quad i=1,2,
\end{equation}
where $\alpha_{i} > -1$. $B(0,1)$ is the unit ball in $\mathbb R^2$ ( we use $B(p,r)$ to denote the ball with centered $p$ and radius
$r$). Throughout of the paper, $h_1^k,h_2^k$ are smooth functions satisfying $h_1^k(0)=h_2^k(0)=1$ and
\begin{equation}\label{ah}
\frac 1C\le h_i^k\le C, \quad \|h_i^k\|_{C^1(B(0,1))}\le C, \quad \mbox{in } B(0,1), \quad  i=1,2.
\end{equation}
For solutions $u^k=(u_1^k,u_2^k)$ we assume:
\begin{equation}\label{asm-u}
\left\{\begin{array}{ll}
 (i): \mbox{ $0$ is the only blowup point of $u^k$},
\\
(ii): |u_i^k(x)-u_i^k(y)|\le C, \quad \forall x,y\in \partial B(0,1),\,\, i=1,2,
\\
(iii): \int_{B(0,1)}h_i^k e^{u_i^k}\le C, \quad i=1,2.
\end{array}
\right.
\end{equation}

For this sequence of blowup solutions we define the local mass by
\begin{equation}\label{12sep2e3}
\sigma_i=\lim_{r\to 0}\lim_{k\to \infty} \frac{1}{2\pi} \int_{B(0,r)} h_i^k e^{u_i^k},\quad i=1,2.
\end{equation}

It is known that $0$ is a blowup point if and only if $(\sigma_{1},\sigma_{2}) \neq (0,0)$. The proof is to use ideas from \cite{bm} and has become standard now. We refer the readers to \cite{lee} for a complete proof. One important property of $(\sigma_1,\sigma_2)$ is the so-called Pohozaev identity (P.I. in short):
\begin{equation}\label{pi-3}
k_{21}\sigma_{1}^{2} + k_{12}k_{21}\sigma_{1}\sigma_{2} + k_{12}\sigma_{2}^{2} = 2k_{21}\mu_{1}\sigma_{1}+2k_{12}\mu_{2}\sigma_{2},
\end{equation}
where $\mu_i=1+\alpha_i$. Take $A_{2}$ as an example, the P.I. is
\begin{equation*}
\sigma_{1}^{2} - \sigma_{1}\sigma_{2} + \sigma_{2}^{2} = 2 \mu_{1}\sigma_{1} + 2\mu_{2}\sigma_{2}.
\end{equation*}

The proof of (\ref{pi-3}) was given in \cite{lwz-apde}. At first sight, (\ref{pi-3}) seems not very useful to determine the local mass. In \cite{lwz-apde} we
initiated an algorithm to calculate all the possible (finitely many) values of local masses. The P.I. plays one of important roles. But the argument there seems not very efficient. In this work we develop further our original approach to sharpen the result:

\begin{thm} \label{thm-main} Suppose $\sigma_1$ and $\sigma_2$ are local masses of a sequence of blowup solutions of (1.4) such that (1.6) holds. Then $\sigma_i$ can be written as
$$\sigma_i=2(N_{i,1}\mu_1+N_{i,2}\mu_2+N_{i,3}),\quad i=1,2,$$
for some $N_{i,1},N_{i,2}, N_{i,3}\in \mathbb{Z}$ ($i=1,2$).
\end{thm}

Theorem 1.3 is proved in section 5 and section 6. In section 5, we give an explicit procedure to calculate the local masses. Take $A_{2}$ system as an example, we start with $\alpha_{1}=0$ and the P.I. gives $\sigma_{2} = 2 \mu_{2}$. With $\sigma_{2} = 2 \mu_{2}$, the P.I. gives $\sigma = 2\mu_{1} + 2\mu_{2}$ and so on. Let $\Gamma(\mu_{1},\mu_{2})$ be the set obtained by the above algorithm. Then $\Gamma(\mu_1,\mu_2)$ is equal to the following set,

\begin{equation*}
\begin{aligned}
&\text{(i)}\ (2\mu_{1},0), (2\mu_{1},2\mu_{1}+2\mu_{2}), (2\mu_{1}+2\mu_{2},2\mu_{1}+2\mu_{2}), (2\mu_1+2\mu_2,2\mu_2),(0,2\mu_2)\ \text{for}\ A_{2},\\[0.1cm]
&\text{(ii)}\ (2\mu_{1},0), (2\mu_{1},4\mu_{1}+2\mu_{2}), (4\mu_{1}+2\mu_{2},4\mu_{1}+2\mu_{2}), (4\mu_{1}+2\mu_{2},4\mu_{1}+4\mu_{2}),\\
& \quad \ \   (0,2\mu_{2}),(2\mu_{1}+2\mu_{2},2\mu_{2}),(2\mu_{1}+2\mu_{2},4\mu_{1}+4\mu_{2}),\ \text{for}\ B_{2},\\[0.1cm]
&\text{(iii)}\ (2\mu_{1},0), (2\mu_{1},6\mu_{1}+2\mu_{2}), (6\mu_{1}+2\mu_{2},6\mu_{1}+2\mu_{2}), (6\mu_{1}+2\mu_{2},12\mu_{1}+6\mu_{2}),\\
& \quad \ \ \ (8\mu_{1}+4\mu_{2},12\mu_{1}+6\mu_{2}),(8\mu_{1}+4\mu_{2},12\mu_{1}+8\mu_{2}), (0,2\mu_{2}), (2\mu_{1}+2\mu_{2},2\mu_{2}),\\
& \quad \ \ \   (2\mu_{1}+2\mu_{2},6\mu_{1}+6\mu_{2}), (6\mu_{1}+4\mu_{2},6\mu_{1}+6\mu_{2}),(6\mu_{1}+4\mu_{2},12\mu_{1}+8\mu_{2}),\ \text{for}\ G_{2}.
\end{aligned}
\end{equation*}
\\

\noindent{\bf Definition 1.4.} {\em
A pair of local masses $(\sigma_1,\sigma_2)\in\Gamma(\mu_1,\mu_2)$ is called special if
\begin{equation*}
(\sigma_1,\sigma_2)=\left\{\begin{array}{ll}
(2\mu_1+2\mu_2,2\mu_1+2\mu_2)~&\mathrm{for}~A_2,\\
(4\mu_1+2\mu_2,4\mu_1+4\mu_2)~&\mathrm{for}~B_2,\\
(8\mu_1+4\mu_2,12\mu_1+8\mu_2)~&\mathrm{for}~G_2.
\end{array}\right.
\end{equation*}}
\\

\indent The analysis of local solutions in \cite{lwz-apde} is a method to pick up these points $\Gamma_{k} = \{ 0,x_{1}^{k}, \cdots , x_{N}^{k} \}$ (if $0$ is a singular point, otherwise $0$ can be deleted from $\Gamma_k$) such that a tiny ball $B(x_{i}^{k},l_{j}^{k})$ can contribute an amount of mass (which is quantized), and the following Harnack-type inequality holds:
\begin{equation}
\label{1.harnack}
u_{i}^{k}(x) + 2\log\text{dist}\ (x,\Sigma_{k}) \leqslant C,\ \forall x \in B(0,1).
\end{equation}

When $\alpha_1=\alpha_2=0$, we can use Theorem 1.3 to calculate all the pairs of even positive integers of (\ref{pi-3}). It turns out the set of solution of (\ref{pi-3}) to be the same as $\Gamma(1,1)$.
\medskip

\noindent{\bf Corollary 1.5.} {\em
Suppose $\alpha_1=\alpha_2=0$. Then $(\sigma_{1},\sigma_{2}) \in \Gamma(1,1)$. Furthermore if $(\sigma_1,\sigma_2)$ is not special, then $\Sigma_k=\{x_1^k\}$ and}
$$u_i^k(x)+2\log|x-x_1^k|\leq C,\quad i=1,2.$$

It is interesting to see whether any pair of the above is really the local masses of some sequence of blowup solutions of (\ref{e-gen}). For $K=A_2$ the existence of such a local solution has been obtained (see \cite{wei} and \cite{lin-yan1}). We remark that parts of Corollary 1.5 was already proved by Jost-Lin-Wang \cite{jostlinwang}, and by the first author and the fourth author in \cite{lin-zhang-b2g2}.

After $\Sigma_k$ is picked up, the difficulty at the next step is how to calculate the mass contributed from outside $B(x^{k}_{j},l_{j}^{k})\ j=1,2, \cdots ,N$. In section 6, we see that the mass outside of this union could be very messy. However, if ($\alpha_{1},\alpha_{2}$) satisfies the $Q$-condition:
\begin{equation*}
[Q]\ \ \alpha_{1},\alpha_{2}\ \textit{and}\ 1\ \textit{are linearly independent over}\ Q. \qquad \qquad \qquad \qquad \qquad \qquad  \qquad
\end{equation*}
Then the result can be stated cleanly as follows.
\medskip

\noindent{\bf Theorem 1.6.}
{\em Suppose ($\alpha_{1},\alpha_{2}$) satisfies the $Q$-condition. Then ($\sigma_{1},\sigma_{2}$) $\in$ $\Gamma$($\mu_{1},\mu_{2}$). Furthermore, the Harnack-type inequality holds:
\begin{equation*}
u^{k}_{i}(x) + 2\log|x| \leqslant C\quad\text{for}\ x \in B(0,1).
\end{equation*}}
\medskip

For (\ref{e-gen}), let $\mu_{1,t}=\alpha_t^1+1$ and $\mu_{2,t}=\alpha_t^2+1$ at a vortex point $p_t\in S$, and define
\begin{equation}
\label{1-gamma}
\Gamma_{i} = \{4\pi(\Sigma_{t \in J}\sigma_{i,t} + n)\ |\ (\sigma_{1,t},\sigma_{2,t}) \in \Gamma(\mu_{1,t},\mu_{2,t}), J \subseteq S,\ n \in \mathbb{N} \cup \{0 \} \}.
\end{equation}
Together with Theorem 1.6, Theorem \ref{main-manifold} can be extended:
\medskip

\noindent{\bf Theorem 1.7.}
{\em Let $h_{i}$ be positive $C^{1}$ functions on $M$, and $K$ be a compact set in $M$. If either both $\alpha_{1}$ and $\alpha_{2}$ are integers or ($\alpha_{1},\alpha_{2}$) satisfies the $Q$-condition and $\rho_{i} \notin \Gamma_{i}$ for $i=1,2$, then there exists a constant $C$ such that
\begin{equation*}
| u_{i}(x) | \leqslant C\ \ \ \forall x \in K.
\end{equation*}}

The organization of this article is as follows. In Section 2 we establish the global mass for the entire solutions of some singular Liouville equation defined in $\mathbb{R}^2$. Then in Section 3 we review some fundamental tools proved in the previous work \cite{lwz-apde}. In section four we present two crucial lemmas, which play the key role in the proof of main results. Then in section 5 and section 6 we discuss the local mass on each bubbling disk centered at $0$ and not at $0$ respectively, thereby we prove all the results.

\section{Totoal mass for liouville equation}
The main purpose of this section is to prove an estimate of the total mass of a solution of the following equation:
\begin{equation}\label{entire-u}
\left\{\begin{array}{ll}
\Delta u+e^u=\sum_{j=1}^N 4\pi \alpha_i\delta_{p_i}, \quad \mbox{ in }\quad \mathbb R^2, \\
\\
\int_{\mathbb R^2}e^u<\infty,
\end{array}
\right.
\end{equation}
where $p_1,...,p_N$ are distinct points in $\mathbb R^2$ and $\alpha_{i} > -1,\ \forall\ 1 \leqslant i \leqslant N$.

\begin{thm}\label{global-energy}
Suppose $u$ is a solution of (\ref{entire-u}) and
$\alpha_1,...,\alpha_N$ are positive integers. Then
$\displaystyle{\frac 1{4\pi}\int_{\mathbb R^2}e^u}$ is an even integer.
\end{thm}

\begin{proof}
It is known that any solution $u$ of (\ref{entire-u}) has the following asymptotic behavior at infinity:
\begin{equation}\label{u-inf}
u(z)=-2\alpha_{\infty}\log |z|+O(1), \quad \alpha_{\infty}>1,
\end{equation}
and $u$ satisfies
\begin{equation}\label{u-integral}
\frac{1}{2\pi}\int_{\mathbb R^2}e^udx=2\sum_{i=1}^N \alpha_i +2\alpha_{\infty}.
\end{equation}
We shall prove that $\alpha_{\infty}+\sum_{i=1}^N\alpha_i$ is an even integer.
A classical Liouville theorem ( see \cite{lin-cheng} ) says that, $u$ can be written as
\begin{equation}\label{liou-1}
u=\log \frac{4|f'(z)|^2}{(1+|f(z)|^2)^2}, \quad z\in \mathbb{R}^{2},
\end{equation}
for some meromorphic function $f$. In general, $f(z)$ is multi-valued and any vertex $p_i$ is a branch point. However if $\alpha_i\in \mathbb N\cup \{0\}$, $f(z)$ is single-valued. Furthermore (\ref{u-inf}) implies that $f(z)$ is meromorphic at infinity. Hence for any solution $u$ of (\ref{entire-u}) there is a meromorphic functon $f$ on $\mathbb S^2=\mathbb C\cup \{\infty\}$ such that (\ref{liou-1}) holds.
Then
\begin{equation*}
\begin{aligned}
4\pi(\sum_{j=0}^N\alpha_j+\alpha_{\infty})=\int_{\mathbb R^2}e^u
&=4\int_{\mathbb R^2}\frac{|f'(z)|^2}{(1+|f(z)|^2)^2}dxdy \\
&=4(deg f)\int_{\mathbb R^2}\frac{d\tilde x d\tilde y}{(1+|w|^2)^2}
= 8\pi(deg f),
\end{aligned}
\end{equation*}
where $deg(f)$ is the degree of $f$ as a map from $\mathbb S^2=\mathbb C\cup \{\infty\}$ onto $\mathbb S^2$, and $w=f(z)=\tilde x+i\tilde y$. Thus we have
$$\sum_{j=0}^N\alpha_j+\alpha_{\infty}=2 deg(f). $$
Theorem \ref{global-energy} is established.
\end{proof}

\begin{thm}\label{glo-ener-2}
Suppose $u$ is a solution of
\begin{equation}\label{non-int-1}
\left\{\begin{array}{ll}
\Delta u+ e^u=4\pi \alpha_0 \delta_{p_0}+\sum_{i=1}^N4\pi \alpha_i\delta_{p_i}, \quad \mbox{ in }\quad \mathbb R^2,\\
\int_{\mathbb R^2}e^u<\infty.
\end{array}
\right.
\end{equation}
where $p_0,p_1,...,p_N$ are distinct points in $\mathbb R^2$ and $\alpha_1,....,\alpha_N$ are positive integers, $\alpha_0>-1$.
Then $\displaystyle{\frac 1{4\pi}\int_{\mathbb R^2}e^u}$ is equal to $2(\alpha_0+1)+2k$ for some $k\in \mathbb{Z}$ or $2k_1$ for some
 $k_1\in \mathbb N$.
\end{thm}

\begin{proof}
As in Theorem \ref{global-energy} there is a developing map $f(z)$ of $u$ such that
\begin{equation}\label{2-1s}
u(z)=\log \frac{4|f'(z)|^2}{(1+|f(z)|^2)^2},\quad z\in \mathbb C.
\end{equation}

On one hand by (\ref{non-int-1}), $u_{zz}-\frac 12 u_z^2$ is a meromorphic function in $\mathbb C\cup \{\infty\}$ because away from the Dirac masses
$$4(u_{zz}-\frac 12 u_z^2)_{\bar z}=-(e^u)_z+u_ze^u=0. $$
By $u(z)=2\alpha_i \log |z-p_i|+O(1)$ near $p_i$ we have
$$u_{zz}-\frac 12 u_z^2=-2\{\sum_{j=0}^N \frac{\alpha_j}2(\frac{\alpha_j}2+1)(z-p_j)^{-2}+A_j(z-p_j)^{-1}+B\}, $$
where $B\in \mathbb C$ is an unknown constant. On the other hand by (\ref{2-1s}), a straightforward computation shows that
\begin{equation}\label{2-2s}
u_{zz}-\frac 12u_z^2=\frac{f'''}{f'}-\frac 32(\frac{f''}{f'})^2.
\end{equation}
Using the Schwarz derivative of $f$:
$$\{f;z\}=\frac{f'''(z)}{f'(z)}-\frac 32(\frac{f''(z)}{f'(z)})^2$$
and letting
$$I(z)=\sum_{j=0}^N\frac{\alpha_j}2(\frac{\alpha_j}2+1)(z-p_j)^{-2}+A_j(z-p_j)^{-1}+B, $$
we write the equation for $f$ as
\begin{equation}\label{add-8-12-1}
\{f,z\}=-2I(z).
\end{equation}
A well known classic theorem (see \cite{w-w}) says that for any two linearly independent solutions $y_1$ and $y_2$ of
\begin{equation}\label{2-3s}
y''(z)=I(z)y(z),
\end{equation}
the ratio $y_2/y_1$ always satisfies
$$\{y_2/y_1;z\}=-2 I(z). $$
By (\ref{add-8-12-1}) and the basic result of the Schwarz derivative, $f(z)$ can be written as the ratio of two linearly independent solutions.
This is how equation (\ref{entire-u}) is related to the complex ODE (\ref{2-3s}). We refer the readers to \cite{chai} for the details.

For the complex ODE (\ref{2-3s}), there is an associated monodromy representation $\rho$ from $\pi_1(\mathbb C\setminus \{p_0,p_1,...,p_{N}\};q)$ to $GL(2;\mathbb C)$
where $q$ is a base point. Note that at any singular point $p_j$ the local exponents are $\frac{\alpha_j}2+1$ and $-\frac{\alpha_j}2$.
So we have
$$\rho_j=\rho(\gamma_j)=C_j\left(\begin{array}{cc}
e^{\pi i \alpha_j} & 0 \\
0 & e^{-\pi i \alpha_j}
\end{array}
\right )C_j^{-1}  ,
$$
where $C_j$ is an invertible matrix, $\gamma_j\in \pi_1(\mathbb C\setminus \{p_0,...,p_N\},q)$ encircles $p_j$ once only, $0\le j\le N$. Then we have
$$\rho_{\infty}\rho_N...\rho_0=I_{2\times 2}. $$
Note that $\rho_j=\pm I$ for $1\le j\le N$. Hence
$$\rho_{\infty}^{-1}=C_0\left(\begin{array}{cc}
e^{\pi\sum_{j=0}^N \alpha_j}  & 0  \\
0 & e^{-\pi\sum_{j=0}^N \alpha_j}
\end{array}
\right ) C_0^{-1} $$
for some constant invertible matrix $C_0$.

On the other hand, the local exponents at $\infty$ can be computed as follows. Recall (\ref{2-3s}) and let $\hat y(z)=y(1/z)$. Then we have
\begin{equation}\label{add-8-12-2}
\hat y''(z)+\frac{2}{z}\hat y'(z)=\hat I(z)\hat y(z),
\end{equation}
where $\hat I(z)=I(1/z)z^{-4}$. Since $I(z)$ is the Schwarz derivative of $f(z)$, by direct computation $\hat I(z)$ is the Schwarz derivative of $f(1/z)$. As before we let
$\hat u(z)=u(1/z)-4\log |z|$. Then $f(1/z)$ is the developing map of $\hat u(z)$. Since
$$\hat u(z)=2(\alpha_{\infty}-2)\log |z|+O(1)~\mbox{ near }~0,$$
(because $u(z)=-2\alpha_{\infty}\log |z|+O(1)$ at infinity), we have
$$\hat I(z)=\frac{\alpha_{\infty}}2(\frac{\alpha_{\infty}}2-1)z^{-2}+\mbox{ higher order terms of $z$}~\mbox{ near } 0. $$
By (\ref{add-8-12-2}) we could prove that the local exponents of (\ref{2-3s}) are $-\frac{\alpha_{\infty}}2$ and $\frac{\alpha_{\infty}}2-1$.
 Hence $e^{\frac{\alpha_{\infty}}2\pi i}$ equals either $e^{i\pi \sum_{j=0}^N \alpha_j}$ or
$e^{-i \pi \sum_{j=0}^N\alpha_j}$, which yields
\begin{equation}\label{k-nn}
\alpha_{\infty}=-\sum_{j=0}^N\alpha_j +2k \quad \mbox{ or } \quad \alpha_{\infty}=\sum_{j=0}^N\alpha_j +2k
\end{equation}
for some $k\in \mathbb Z$.
Since
$$\frac 1{4\pi}\int_{\mathbb R^2}e^u=\sum_{j=0}^N\alpha_j+\alpha_{\infty}, $$
we see that either $\frac 1{4\pi}\int_{\mathbb R^2}e^u=2k$ if the first case holds or
$\frac 1{4\pi}\int_{\mathbb R^2}e^u=2(\alpha_0+1)+2k'$ for $k'=2\sum_{i=1}^N\alpha_i+2k-2$ if the second case holds.
\end{proof}

\begin{rem} After Theorem \ref{global-energy} and Theorem \ref{glo-ener-2} haven been proved, we found a stronger version of both theorems in \cite{erem-3}. Because we only need the present form of both theorems, we include our proofs here to make the paper more self-contained.
\end{rem}

\section{Review of Bubbling Analysis From a Selection Process}

Let $u^k=(u_1^k,u_2^k)$ be solutions of (\ref{toda-b2g2}) such that (\ref{asm-u}) holds. In this section we review the process to select a set $\Sigma_k=\{0,x_1^k,...,x_n^k\}$ and balls $B(x_i^k,l_k)$ such that $u^k$ has nonzero local masses in $B(x_{i}^{k},l_{k})$. This selection process was first carried out in \cite{lwz-apde}. We briefly review it below.

The set $\Sigma_k$ is constructed by induction. If (\ref{toda-b2g2}) has no singularity, we start with $\Sigma_k=\emptyset$. If (\ref{toda-b2g2}) has a singularity, we start with $\Sigma_k=\{0\}$. By
induction suppose $\Sigma_k$ consists of $\{0,x_1^k,...,x_{m-1}^k\}$. Then we consider
\begin{equation}\label{se-eq}
\max_{x\in B_1} \bigg ( u_i^k(x)+2\log\mbox{dist}(x,\Sigma_k) \bigg ).
\end{equation}
If the maximum is bounded from above independent of $k$, the process stops and $\Sigma_k$ is exactly equal to $\{0,x_1^k,...,x_{m-1}^k\}$. However if the maximum tends to infinity, let $q_k$ be where (\ref{se-eq}) is achieved and we set
$$d_k=\frac 12\mbox{dist}(q_k,\Sigma_k) $$
and
$$S_i^k(x)=u_i^k(x)+2\log (d_k-|x-q_k|)\quad \mbox{ in }\quad B(q_k,d_k), \quad i=1,2. $$
Suppose $i_0$ is the component that attains
\begin{equation}\label{se-eq-2}
\max_i\max_{x\in \bar B(q_k,d_k)}S_i^k
\end{equation}
at $p_k$. Then we set
$$\tilde l_k=\frac 12(d_k-|p_k-q_k|)$$
and scale $u_i^k$ by
\begin{equation}
v_i^k(y)=u_i^k(p_k+e^{-\frac 12u_{i_0}^k(p_k)}y)-u_{i_0}^k(p_k), \quad \mbox{ for }\quad |y|\le R_k \doteqdot e^{\frac 12 u_{i_0}^k(p_k)} \tilde l_k.
\end{equation}
It can be shown that $R_k\to \infty$ and $v_i^k$ is bounded from above over any fixed compact subset of
$\mathbb R^2$. Thus by passing to a subsequence $v_i^k$ satisfies one of the following two alternatives:

(a)\ $(v_1^k,v_{2}^{k})$ converges in $C^2_{loc}(\mathbb R^2)$ to
$(v_1,v_{2})$ which satisfies
\begin{equation}
\Delta v_i+\sum_{j\in I} k_{ij} e^{v_j}=0~\mathrm{in}~~\mathbb R^2, \quad i\in I=\{1,2\}.
\end{equation}

(b)\ Either $v_1^k$ converges to
\begin{equation}
\Delta v_1+2e^{v_1}=0~\mathrm{in}~\mathbb R^2
\end{equation}
and $v_2^k\to -\infty$ over any fixed compact subset of $\mathbb R^2$ or $v_2^k$ converges to $\Delta v_2+2e^{v_2}=0$ in $\mathbb R^2$ and $v_1^k\to -\infty$ over any fixed compact subset of $\mathbb R^2$.

\medskip

Therefore in either case, we could choose $l_k^*\to \infty$  such that
\begin{equation}\label{add-e-8-1}
v_i^k(y)+2\log |y|\le C, \quad \mbox{ for } i=1,2\ \text{and}\ |y| \leqslant l_k^*
\end{equation}
and
$$\int_{B(0,l_k^*)}h_i^ke^{v_i^k}dy=\int_{\mathbb R^2}e^{v_i(y)}+o(1). $$
By scaling back to $u_i^k$, we add $p_k$ in $\Sigma_k$ with $l_k=e^{-\frac 12 u_{i_0}^k(p_k)}l_k^*$. We can continue in this way until the Harnack-type inequality (1.9) holds.

We summarize what the selection process has done in the following proposition ( a detailed proof for a more general case can be found in Proposition 2.1 of \cite{lwz-apde}):
\medskip

\noindent{\bf Proposition 3A.}
{\em Let $u^k$ be described as above.  Then there exist a finite set
$\Sigma_k:=\{0,x_1^k,....,x_m^k\}$ (if $0$ is not a singular point, then $0$ can be deleted from $\Sigma_k$) and
positive numbers $l_1^k,...,l_m^k\to 0$ as $k \rightarrow \infty$ such that the followings hold:
\begin{enumerate}
\item There exists $C_1>0$ independent of $k$ such that (\ref{1.harnack}) holds.
\item In $B(x_j^k,l_j^k)$ ($j=1,..,m$), let $R_{j,k}=e^{\frac 12u_{i_0}^k(x_j^k)}l_{j}^{k}$, $u_{i_0}^k(x_j^k)=\max_iu_i^k(x_j^k)$ and
\begin{equation}
\label{scal-v}
v_i^k(y)=u_i^k(x_j^k+e^{-\frac12u_{i_0}^k(x_j^k)}y)-u_{i_0}^k(x_j^k)
\end{equation}
for $|y|\le R_{j,k}$, then
$v^k=(v_1^k,v_2^k)$ satisfies either (a) or (b).
\item $B(x_{j}^{k},l_{j}^{k}) \cap\ B(x_{i}^{k},l_{i}^{k}) =\emptyset.$
\end{enumerate}}
\medskip

The inequality (\ref{1.harnack}) is a Harnack type inequality, because it implies the following result
\medskip

\noindent{\bf Proposition 3B.}
{\em Suppose $u^k$ satisfies (\ref{toda-b2g2}) in $B(x_0,2r_{k})$ such that
$$u_i^k(x)+2\log |x-x_0|\le C,\quad \mbox{ for }~x\in B(x_0,r_{k}). $$
Then
\begin{equation}
\label{har-2}
|u_i^k(x_1)-u_i^k(x_2)|\le C_0,\quad \mbox{ for } \frac 12\le \frac{|x_1-x_0|}{|x_2-x_0|}\le 2 \mbox{ and } x_1,x_2\in B(x_0,r_{k}).
\end{equation}}

The proof of Proposition 3B is standard (see \cite[Lemma 2.4]{lwz-apde}), so we omit it here. Let $x_l^k\in \Sigma_k$ and $\tau_l^k=\frac 12
\mbox{dist}(x_l^k,\Sigma_k\setminus \{x_l^k\})$, then (\ref{har-2}) implies
\begin{equation}\label{like-a}
u_i^k(x)=\bar
u_{x_l^k,i}^k(r)+O(1), \quad x \in B(x_{l}^{k},\tau_{l}^{k}),
\end{equation}
where $r=|x_l^k-x|$ and $\bar u_{x_l^k,i}^k$ is the average of $u_i^k$ on $\partial B(x_l^k,r)$:
\begin{equation}\label{uxi-a}
\bar u_{x_l^{k},i}^k(r)=\frac 1{2\pi r}\int_{\partial B(x_l^{k}, r)}u_i^k dS,
\end{equation}
and $O(1)$ is independent of $r$ and $k$.

Next we introduce the notions of slow decay or fast decay in our bubbling analysis.
\begin{Def}\label{f-s-decay}
We say $u_i^k$ has fast decay at $x\in B(x_0,r_k)$
if along a subsequence,
$$u_i^k(x)+2 \log |x-x_0|\le -N_k,\quad \mbox{ for }~x\in \partial B(x_0,r_k) $$
for some $N_k\to \infty$ and $u_i^k$ is called to have  slow-decay if there is a constant $C$ independent
of $k$ and
$$u_i^k(x)+2 \log |x-x_0|\ge -C,\quad \mbox{ for }~x\in \partial B(x_0,r_k). $$
\end{Def}

Fast decay is very important for
evaluating Pohozaev identities. The following proposition is a direct consequence of \cite[Proposition 3.1]{lwz-apde} and it says if both components are fast-decay on the boundary, Pohozaev identity holds for the local masses.

In the following proposition, we let $B=B(x^k,r_k)$. If $x^k\neq0,$ then we assume $0\notin B(x^k,2r_k).$
\medskip

\noindent{\bf Proposition 3C.}
{\em Suppose both $u_1^k,u_2^k$ have fast decay on $\partial B$, where $B$ is given above. Then $(\sigma_{1},\sigma_{2})$ satisfies the P.I.(1.8), where $$\sigma_i=\lim_{k\to 0}\frac{1}{2\pi}\int_{B}h_i^ke^{u_i^k},~i=1,2.$$}

\indent The proof of Proposition 3C requires some delicate analysis. We refer the readers to \cite[Proposition 3.1]{lwz-apde} for the proofs. The P.I. plays an important role in our analysis later.
\medskip

\section{Two Lemmas}
In this section, we will prove two crucial lemmas which play the key role in section 5 and 6. For Lemma \ref{le4.1}, we assume

\noindent (i). The Harnack inequality
$$u_i^k(x)+2\log |x|\le C,~\mbox{ for }~\frac 12 l_k\le |x|\le 2s_k,~\mathrm{and}~i=1,2. $$
\noindent (ii). Both $u_i^k$ have fast-decay on $\partial B(0,l_k)$ and $\sigma_i^k(B(0,l_k))=\sigma_i+o(1)$ for $i=1,2$, where $\sigma_i=\lim\limits_{r\rightarrow 0}\lim\limits_{k \rightarrow \infty}\sigma_{i}(B(0,rs_{k})),\ i=1,2$.

\noindent (iii). One of $u_i^k$ has slow-decay on $\partial B(0,s_k)$.

\begin{lem}
\label{le4.1}
\noindent(a). Assume (i) and (ii), If $u_{i}^{k}$ has slow-decay on $\partial B(0,s_k)$, then $$2\mu_{i}-\sum\limits_{j=1}^{2}k_{ij}\sigma_{j}\geqslant0.$$
\noindent(b). Assume (i), (ii) and (iii), then the other component has fast decay on $\partial B(0,s_k)$.
\end{lem}

\begin{proof}
(a) Suppose $u_i^k$ have slow decay on $\partial B(0,s_k)$, then
the following scaling
$$v_j^k(y)=u_j^k(s_ky)+2\log s_k, \quad j=1,2, \quad \mbox{ for } y\in B_2 $$
gives
$$\Delta v_j^k(y)+\sum_{l=1}^2 k_{jl} h_l^k(s_ky) e^{v_l^k(y)}=4\pi \alpha_i^k \delta_0, \quad \mbox{ in } y\in B_2. $$

If the other component also has slow-decay on $\partial B(0,s_k)$, then $(v_{1}^k,v_{2}^k)$ coverages to $(v_1,v_2)$ which satisfies
\begin{equation}
\label{fourvj}
\Delta v_{j}(y) + \sum\limits_{j=1}^{2}k_{jl}e^{v_j}=0,\ \text{in}\ B_2 \backslash \{0 \},~j=1,2.
\end{equation}
If the other component has fast-decay on $\partial B(0,s_k)$, then $v_{i}^{k}(y)$ coverages to $v_{i}(y)$ and $v_{j}(y) \rightarrow -\infty$, $j\neq i$. Furthermore, $v_{i}(y)$ satisfies
\begin{equation}
\label{fourvi}
\Delta v_{i}(y)  + 2e^{v_{i}} = 0\ \text{in}\ B_2 \backslash \{0 \}.
\end{equation}
For any $r>0$,
\begin{equation*}
\begin{aligned}
\int_{\partial B(0,r)}\frac{\partial v_{i}(y)}{\partial v} dS &= \lim\limits_{k\rightarrow\infty}(4\pi\alpha^{k}_{i}-\sum\limits_{j=1}^{2}\int_{B(0,r)}k_{ij}h^{k}_{j}e^{v^{k}_j}dy) \\
&= 4\pi\alpha_{i} - 2\pi \sum\limits_{j=1}^{2}k_{ij}\sigma_{j} + o(1) \doteqdot 4\pi\beta_{i}+o(1),
\end{aligned}
\end{equation*}
which implies RHS of both (\ref{fourvj}) and (\ref{fourvi}) should be replaced by $4\pi\beta_{i}\delta(0)$ as an equation defined in $B_2$. It is known that if $\beta_i < -1$, either (\ref{fourvj}) or (\ref{fourvi}) has no solutions. Hence $\alpha_i - \frac{1}{2}\sum k_{ij}\sigma_{j}\geqslant -1$ and then (a) is proved.
\vspace{0.2cm}

\noindent (b) Since both components have fast decay on $\partial B(0,l_k)$, the pair $(\sigma_1,\sigma_2)$ satisfies the P.I. (\ref{pi-3}). By a simple manipulation, the P.I. (\ref{pi-3}) can be written as
\begin{equation}
\label{fourpi}
k_{21}\sigma_{1}(4\mu_{1}-k_{12}\sigma_2-k_{11}\sigma_1) + k_{12}\sigma_2(4\mu_2-k_{21}\sigma_{1}-k_{22}\sigma_2)=0
\end{equation}
Note by (a)
\begin{equation*}
4\mu_{i} - \sum\limits_{l=1}^{2}k_{il}\sigma_{l} > 2\mu_{i} - \sum\limits_{l=1}^{2}k_{il}\sigma_{l} \geqslant 0.
\end{equation*}
Hence for $j \neq i$
\begin{equation*}
2\mu_{j} - \sum\limits_{j=1}^{2}k_{il}\sigma_{l} < 4\mu_{j}-\sum\limits_{j=1}^{2}k_{il}\sigma_{l} < 0 ,
\end{equation*}
where the last inequality is due to (\ref{fourpi}). By (a) again, $u_{j}^k$ can not have slow-decay on $\partial B(0,s_k)$.
\end{proof}

Our second lemma is about the fast-decay.
\begin{lem}
\label{le4.2}
Suppose the Harnack-type inequality holds for both components over $r \in [\frac{l_k}{2},2s_k]$. If $u_i^k$ is fast-decaying on $r \in [l_k,s_k]$, then
\begin{equation*}
\sigma^{k}_{i}(B(0,s_k)) = \sigma^{k}_{i}(B(0,l_k))  + o(1).
\end{equation*}
\end{lem}
\begin{proof}
Obviously the conclusion holds easily if $s_k/l_k \leqslant C$. So we assume $s_k/l_k \rightarrow +\infty$. The Harnack-type inequality implies $u_{i}^{k}(x) = \overline{u}_{i}^{k}(r) + o(1)$ for $\frac{1}{2}l_k \leqslant |x| \leqslant 2 s_k$. Thus we have from (\ref{toda-b2g2}) that
\begin{equation*}
\frac{d}{dr}(\overline{u}_{i}^{k}(r)+2\log\ r) =\frac{2\mu_i -\sum_{j=1}^2 k_{ij}\sigma_{j}^{k}(r)}{r}, \quad l_k \leqslant r \leqslant s_k,~i=1,2,
\end{equation*}
where $\sigma_j^{k}(r)=\sigma_j^k(B(0,r))$ and $\sigma_j=\lim_{k\to+\infty}\sigma_j^k(l_k),~j=1,2$.

By our assumption, the P.I. holds at $l_k$, which implies at least one component satisfies
\begin{equation*}
4\mu_l - \sum k_{lj}\sigma_{j}^{k}(l_k) \geqslant 2\mu_l+o(1)> 0.
\end{equation*}
Thus,
\begin{equation}
\label{fourd}
\frac{d}{dr}(\overline{u}_{l}^{(k)}(r)+2\log\ r) \leqslant -\frac{2\mu_l +o(1)}{r}\ \text{at}\ r= l_k.
\end{equation}
Suppose $r_k\in[l_k,s_k]$ is the largest $r$ such that
\begin{equation}
\label{fourd1}
\frac{d}{dr}(\overline{u}_{l}^{(k)}(r)+2\log\ r) \leqslant -\frac{\mu_l}{r}\ \text{for}\ r \in [l_k,r_k],
\end{equation}
thus, either the identity holds at $r=r_k$ or $r_k=s_k.$ For simplicity, we let $\varepsilon =\mu_l$. By integrating (4.4) from $l_k$ up to $r \leqslant r_k$, we have
\begin{equation*}
\overline{u}_{l}^{(k)}(r) + 2\log\ r \leqslant \overline{u}_{l}^{(k)}(l_k) + 2 \log (l_k) + \epsilon\log(\frac{l_k}{r}),
\end{equation*}
that is for $|x|=r$,
\begin{equation*}
e^{u_{l}^{k}(x)} \leqslant O(1)e^{\overline{u}_{l}^{k}(r)} \leqslant e^{-N_k}l^{\epsilon}_{k}r^{-(2+\epsilon)},
\end{equation*}
where we used $\overline{u}_{l}^{(k)}(l_k)+2\log l_k \leqslant -N_k$ by the assumption of fast-decay. Thus
\begin{equation*}
\int_{l_k \leqslant |x| \leqslant r_k} e^{u^{k}_{l}(x)}dx \leqslant 2\pi e^{-N_k}l^{\varepsilon}_k \int^{r_k}_{l_k}r^{-(1+\varepsilon)}dr = 2\pi\frac{e^{-N_k}}{\varepsilon} \rightarrow 0
\end{equation*}
as $k \rightarrow + \infty$. Hence
\begin{equation}
\label{sigma.lk}
\sigma_{l}^{k}(r_k) = \sigma^{k}_{l}(l_k)+o(1).
\end{equation}
%

If both components are fast decaying on $r\in [l_k,r_k]$, then $\lim\limits_{k \rightarrow+ \infty}(\sigma^{k}_{1}(r_k),\sigma^{k}_{2}(r_k))=(\hat{\sigma}_1,\hat{\sigma}_2)$ also satisfies the P.I.(\ref{pi-3}). If $\hat{\sigma}_j>\sigma_j$, then $j\neq l$ by \eqref{sigma.lk}. We choose $r_k^*\leq r_k$ such that $\sigma_j(r_k^*)=\sigma_j^k(l_k)+\varepsilon_0$ for small $\varepsilon_0$, and let $\sigma_j^*=\lim_{k\to0}\sigma_j(r_k^*).$ Then $\sigma_j^*$ and $\sigma_l$ satisfies the P.I.(\ref{pi-3}) and it yields a contradiction provided $\varepsilon_0$ is small. Thus, we have $\sigma_m^k(r_k)=\sigma_m^k(l_k)+o(1),~m=1,2$. Then (\ref{fourd}) holds at $r=r_k$ which implies $r_k=s_k$, and Lemma \ref{le4.2} is proved in this case.

If one of the components cannot have fast decay on $[l_k,r_k]$. We have $l=i$ and $u_j^k,j\neq i$,  has slow decay on $\partial B(0,r_k^*)$ for some $r_k^*\leq r_k$. If $s_k/r_k\leq C$, then \eqref{sigma.lk} implies the lemma. If $s_k/r_k\to+\infty$, then by the scaling of $u_j^k$ at $r=r_k^*$, the standard argument implies that there is a sequence of $r_k^*\ll\tilde{r}_k=R_kr_k^*\ll s_k$ such that both components have fast decay on $\tilde{r}_k$ and
$$\sigma_i^k(\tilde{r}_k)=\sigma_i(r_k^*)+o(1)=\sigma_i(l_k)+o(1),~\mathrm{and}~\sigma_j^k(\tilde{r}_k)\geq\sigma_j^k(l_k)+\varepsilon_0$$
for some $j\neq i$ and $\varepsilon_0>0$. Therefore the assumption of Lemma \ref{le4.2} holds at $r\in[\tilde{r}_k,s_k]$. Then we
repeat the argument starting from (\ref{fourd}) and the lemma can be proved in a finite steps.
\end{proof}

\noindent {\bf Remark 4.3.} {\em Both lemmas will be used in section 6 (and section 5) for the case with singularity at $0$ (and without singularity at $0$).}
\bigskip

\section{Local mass on the bubbling disk centered at $x_i^k\neq 0$}
\subsection{}

In this subsection we study the local behavior of $u^k$ near $x_l^k$ where $x_{l}^{k} \neq 0$. For simplicity, we use $x^k$ instead of $x_l^k$ and
$\bar u_i^k(r)$ rather than $\bar u_{x^k,i}(r)$. Let
$$\tau^k=\frac 12\mathrm{ dist}(x^k,\Sigma_k\setminus \{x^k\}) \quad
\sigma_i^k(r)=\frac 1{2\pi}\int_{B(x^k,r)}h_i^ke^{u_i^k},\quad i=1,2. $$
By Proposition 3A, $l_k\le \tau^k$. Clearly $u^k=(u_1^k,u_2^k)$ satisfies
$$\Delta u_i^k+\sum_j k_{ij}h_j^k e^{u_j^k}=0,\quad \mbox{ in }\quad B(x^k,\tau^k). $$

For a sequence $s_k,$ we define
\begin{equation}
\label{sigma-def}
\hat{\sigma_{i}}(s_{k})=\left\{\begin{array}{ll}
\lim\limits_{k \rightarrow +\infty}\sigma_{i}^{k}(s_k)\ \text{if}\ u^{k}_{i}\ \text{has fast decay on}\ \partial B(x^{k},s_k),\\
\\
\lim\limits_{r \rightarrow 0}\lim\limits_{k \rightarrow +\infty}\sigma_{i}^{k}(rs_k)\ \text{if}\ u^{k}_{i}\ \text{has slow decay on}\ \partial B(x^{k},s_k),
\end{array}\right.
\end{equation}
Recall that both $u_{i}^{k}$ have fast decay on $\partial B(x^{k},l_{k})$ (see (3.4)). This is the starting point of the following proposition, which is a special case of Proposition \ref{loc-mass-2} below.

In Proposition \ref{local-mass-type-1}, $(\mu_1,\mu_2)$ will be $(1,1)$ in both lemmas of section 4.

\begin{prop}
\label{local-mass-type-1}
Let $u^k=(u_1^k,u_2^k)$ be the solutions of (\ref{toda-b2g2}) satisfying (\ref{asm-u}) and $\hat{\sigma}_i(s_k)$ be defined in \eqref{sigma-def}, the followings hold:
\begin{enumerate}
  \item At least one component $u^k$ has fast decay on $\partial B(x^k,\tau^{k})$,
  \item $(\hat{\sigma_{1}}(\tau_{k}),\hat{\sigma_{2}}(\tau_{k}))$ satisfies the P.I.(\ref{pi-3}) with $\mu_1=\mu_2=1$,
  \item $(\hat{\sigma_{1}}(\tau_{k}),\hat{\sigma_{2}}(\tau_{k})) \in \Gamma(1,1)$.
\end{enumerate}
\end{prop}

\begin{proof}
If $\tau_k/l_k \leqslant C$, (1)-(3) holds obviously for $\tau_{k}$. So we assume $\tau_k/l_k \rightarrow + \infty$. First we remark that if $u^k$ is fully bubbling in $B(x^k,l_k)$ (i,e, (a) in Proposition 3A holds), $(\hat{\sigma}_1(l_k),\hat{\sigma}_2(l_k))$ is special (see Definition 1.4) and satisfies $$2\mu_i-\sum_{j=1}^2k_{ij}\hat\sigma_j(l_k)<0,i=1,2.$$ Then by Lemma \ref{le4.1}, both $u_{i}^{k}$ have fast decay on $\partial B(0,\tau^{k})$ and Proposition \ref{local-mass-type-1} follows immediately.

Now we assume $v_{i}^k$ defined in \eqref{scal-v} and satisfies case(b) in Proposition 3A. From (3.4), we already knew that both components have fast decay at $r=l_k$. If both components remain fast decay as $r$ increases from $l_k$ to $\tau_k$, Lemma \ref{le4.2} implies
$$\sigma_1^k(\tau_k)=\sigma_1^k(l_k)+o(1),~\quad\sigma_2^k(\tau_k)=\sigma_2^k(l_k)+o(1)$$
and we are done. So we only consider the case that at least one component changes to a slow decay component. For simplicity, we assume that $u_1^k$ changes to slow decay for some $r_k \gg l_k$. By Lemma \ref{le4.2},
\begin{equation*}
\sigma_1^{k}(B(x^k,r_k)) \geqslant \sigma_{1}(B(x^{k},l_{k})) + c_0,~\mathrm{for~some}~c_0 > 0.
\end{equation*}
We might choose $s_k \leqslant r_k$ such that
$$\sigma_{1}^{k}(B(x^{k},s_{k}))=\sigma_{1}^{k}(B(x^{k},l_{k}))+\varepsilon_{0},$$
and
$$\sigma_{1}^{k}(B(x^{k},r)) < \sigma_{1}^{k}(B(x^{k},l_{k}))+\varepsilon_{0}\ \forall\ r < s_{k},$$
where $\varepsilon_{0}< \frac{c_0}{2}$ is small.

Then Lemma \ref{le4.1} and Lemma \ref{le4.2} together implies $u_1^k$ has slow-decay on $\partial B(x^{k},s_{k})$ and $u_2^k$ has fast decay on $\partial B(x^k,s_k)$ with
$$\hat\sigma_{1}(s_k)=\sigma_{1}^{k}(l_{k})+o(1)~\mathrm{and}~\hat{\sigma_2}(s_k)=\sigma_{2}^{k}(l_k)+o(1).$$
Let $v_{i}^{k}(y)=u_{i}^k(x^k+s_ky)+2\log\ s_k$. If $\tau_k/s_k\leq C$ there is nothing to prove. So we assume $\tau_k/s_k\to\infty$. Then $v_{1}^{k}(y)$ converges to $v_{1}(y)$ and $v_{2}^{k}(y) \rightarrow -\infty$ in any compact set of $\mathbb{R}^{2}$ as $k \rightarrow + \infty$ and $v_{1}(y)$ satisfies
\begin{equation}
\Delta v_1 + 2e^{v_1} = -2\pi \sum (k_{1j}\hat{\sigma_{j}}(l_k))\delta(0)\ \text{in}\ \mathbb{R}^{2}.
\end{equation}
Hence there is a sequence $N_{k}^{*}$ and $N_k^* \rightarrow +\infty$ as $k \to +\infty$ such that $N_k^*s_k\leq\tau_k$ and
\begin{equation*}
\int_{B(0,N_{k}^{*})} e^{v_1} dy = \int_{\mathbb{R}^{2}} e^{v_1} dy + o(1),\ \text{and}
\end{equation*}
both $v_{i}^{k}(y)+2\log |y| \leqslant -N_{k}$ for $|y|=N_{k}^{*}$. Scaling back to $u_{i}^{k}$, we obtain that $u_{i}^{k},~i=1,2,$ have fast-decay on $\partial B(x^k,N_{k}^*s_{k})$.

We could use the classification theorem of Prajapat and Tarantello \cite{prajapat} to calculate the total mass of $v_1$, but instead we use the P.I.(\ref{pi-3}) to compute it. We know that both $(\hat{\sigma_{1}}(l_k),\hat{\sigma_{2}}(l_k))$ and $(\hat{\sigma_{1}}(N_{k}^*s_k),\hat{\sigma_{2}}(N_{k}^*s_{k}))$ satisfy P.I. and $\hat{\sigma}_{2}(N_{k}^*s_k)=\hat{\sigma_2}(l_k)$ by Lemma \ref{le4.2}. With a fixed $\sigma_2 = \hat{\sigma_2}(l_k)$, the equation P.I. (\ref{pi-3}) is a quadratic polynomial in $\sigma_1$, then $\hat{\sigma_1}(l_k)$ and $\hat{\sigma_1}(N_{k}^*s_{k})$ are two roots of the polynomial. From it, we could easily calculate $\hat{\sigma_1}(N_{k}^*s_{k})$.

By a direct computation, we have
$$(\hat{\sigma_1}(N_{k}^*s_k),\hat{\sigma_2}(N_{k}^*s_k)) \in \Gamma(1,1)~\mathrm{if}~(\hat{\sigma_1}(l_k),\hat{\sigma_2}(l_k))\in \Gamma(1,1).$$
Thus (1)-(3) hold at $r=N_{k}^*s_{k}$. By denoting $N_{k}^*s_k$ as $l_k$, we could repeat the same argument until $\tau_k/l_k \leqslant C$. Hence Proposition \ref{local-mass-type-1} is proved.
\end{proof}

\subsection{Local mass in a group that does not contain $0$}

In this subsection we collect some $x_i^k\in \Sigma_k$ into a group $S$, a subset of $\Sigma_k$ satisfying the following $S$-condition:
\medskip

\noindent(1). $0\not \in S$ and $|S|\ge 2$.

\noindent(2). If $|S|\ge 3$ and $x_i^k$, $x_j^k$, $x_l^k$ are three distinct elements in $S$, then
$$\mbox{dist}(x_i^k,x_j^k)\le C\mbox{dist}(x_j^k,x_l^k)$$
\indent \quad for some constant $C$ independent of $k$.

\noindent(3). For any $x_m^k\in \Sigma_k\setminus S$, the ratio $\mbox{dist}(x_m^k,S)/\mbox{dist}(x_i^k,x_j^k)\to \infty$ as $k\to \infty$ where $x_i^k,x_j^k\in S$.

We write $S$ as
$S=\{x_1^k,...,x_m^k\}$ and let
\begin{equation}
\label{fivelk}
l^k(S)=2\max_{1\le j\le m}\mbox{dist}(x_1^k,x_j^k).
\end{equation}
Recall $\tau_{l,k}=\frac 12\mbox{dist}(x_l^k,\Sigma_k\setminus \{x_l^k\})$, by (2) above we have $l^k(S)\sim \tau_i^k$ for $1\le i\le m$. Let
$$\tau_S^k=\frac 12\mbox{dist}(x_1^k,\Sigma_k\setminus S) $$
then by (3) above $\tau_S^k/\tau_i^k\to \infty$ for any $x_i^k\in S$.

By Proposition \ref{local-mass-type-1}, we know that at least one of $u_i^k$ has fast decay on $\partial B(x_1^k,\tau_1^k)$. Suppose $u_1^k$ has fast decay on $\partial B(x_1^k,\tau_1^k)$.
Then
\begin{equation}
\label{u1k-fd}
\mbox{ $u_1^k$ has fast decay on $\partial B(x_1^k,l^k(S))$,}
\end{equation}
and Proposition \ref{local-mass-type-1} implies
\begin{align*}
\sigma_1^k(B(x_1^k,l^k(S)))
&=\frac 1{2\pi}\int_{B(x_1^k,l^k(S))}h_1^ke^{u_1^k}dx\\
&=\frac 1{2\pi}\int_{\cup_{j=1}^mB(x_j^k,\tau_j^k)}h_1^ke^{u_1^k}+\frac 1{2\pi}\int_{B(x_1^k,l^k(S))\setminus (\cup_{j=1}^mB(x_j^k,\tau_j^k))}h_1^ke^{u_1^k}.
\end{align*}
Since $u_1^k$ has fast decay outside of $B(x_j^k,\tau_j^k)$, we have
$$e^{u_1^k(x)}\le o(1)\max_{j}\{|x-x_j^k|^{-2}\},~\mathrm{for}~x\notin\bigcup_{j=1}^kB(x_j^k,\tau_j^k)$$
and the second integral is $o(1)$. Hence by Proposition \ref{local-mass-type-1},
\begin{equation}\label{52-e1}
\sigma_1^k(B(x_1^k,l^k(S)))=2m_1+o(1) \quad \mbox{ for some } m_1\in \mathbb N\cup\{0\}.
\end{equation}

Similarly if $u_2^k$ has fast decay on $\partial B(x_1^k,\tau_1^k)$, we have
\begin{equation}\label{52-e2}
\sigma_2^k(B(x_1^k,l^k(S)))=2m_2+o(1) \quad \mbox{ for some } m_2\in \mathbb N\cup\{0\}.
\end{equation}

If $u_2^k$ has slow decay on $\partial B(x_1^k,\tau_1^k)$, then it is easy to see that $u_2^k$ has slow decay on $\partial B(x_j^k,\tau_j^k)$.  By Proposition \ref{local-mass-type-1} we denote $n_{i,j}\in \mathbb N$ by
$$2n_{i,j}=\lim_{r\to 0}\lim_{k\to \infty} \sigma_i^k(B(x_j^k,r\tau_j^k)), \quad 1\le j\le m, \quad i=1,2. $$
Set $\hat n_{i,j}$ by
$$\hat n_{i,j}=-\sum_{l=1}^2k_{il}n_{l,j}. $$
Then the slow decay of $u_2^k$ on $\partial B(x_j^k,\tau_j^k)$ implies $1+\hat n_{i,j}>0$. Since $\hat n_{i,j}\in \mathbb Z$ we have $\hat n_{i,j}\ge 0$.

Furthermore, if we scale $u^k$ by
$$v_i^k(y)=u_i^k(x_1^k+l^k(S)y)+2\log l^k(S), \quad i=1,2, $$
the sequence $v_2^k$ would converge to $v_2(y)$ and $v_1^k$ tends to $-\infty$ over any compact subset of $\mathbb R^2\setminus \{0\}$. Then $v_2$ satisfies
\begin{equation}
\label{52-e3-1}
\Delta v_2(y)+2e^{v_2(y)}=4\pi\sum_{j=1}^m \hat n_{i,j}\delta_{p_j} \quad \mbox{ in } \quad \mathbb R^2.
\end{equation}
where $p_j=\lim_{k\to \infty}(x_j^k-x_1^k)/l^k(S)$.
By Theorem \ref{global-energy}
$$ \frac 1{2\pi}\int_{\mathbb R^2}e^{v_2}=2N, \quad \mbox{ for some } N\in \mathbb N. $$
Thus using the argument in Proposition \ref{local-mass-type-1}, we conclude that there is a sequence of $N_k^*\to \infty$ such that both $u_i^k$ ($i=1,2$) have fast decay on $\partial B(x_1^k,N_k^*l^k(S))$ and $\sigma_i^k(B(x_1^k,N_k^*l^k(S)))=2m_i+o(1)$. Denote $N_k^*l^k(S)$ by $l_k$ for simplicity, then we see that (\ref{52-e1}) and (\ref{52-e2}) hold at $l_k$. Then by using Lemma \ref{le4.1} and Lemma \ref{le4.2} we could continue this process to obtain the following conclusion:
\begin{equation}
\label{52-e3}
\mbox{At least one component of}~u^k~\mbox{has fast decay on}~\partial B(x_1^k,\tau_S^k).
\end{equation}
Let $\hat{\sigma}_i^k(B(x_1^k,\tau_s^k))$ be defined as in (\ref{sigma-def}). Then
\begin{equation}
\label{52-e4}
\hat{\sigma}_i^k(B(x_1^k,\tau_S^k))=2m_i(S)+o(1), \quad \mbox{where } m_i(S)\in \mathbb{N}\cup\{0\},
\end{equation}
and the pair $(2m_1(S),2m_2(S))$ satisfies the P.I.(\ref{pi-3}).
\medskip

Denote the group $S$ by $S_1$. Based on this procedure, we could continue to select a new group $S_2$ such that $S$-condition holds except we have to modify condition-(2). In (2), we consider $S_1$ as an single point as long as we compare the distance of distinct elements in $S_2$.

Set
$$\tau_{S_2}^k=\frac 12\mbox{dist}(x_1^k,\Sigma_k\setminus S_2 ),\quad \mbox{for } \quad x_1^k\in S_2. $$
Then we follow the same argument as above to obtain the same conclusion as (\ref{52-e3})-(\ref{52-e4}).

If equation (\ref{toda-b2g2}) does not contain singularity, the final step is to collect all $x_i^k$ into one single biggest group and (\ref{52-e3})-(\ref{52-e4}) hold.  Then we get $(\sigma_1,\sigma_2)=(2m_1,2m_2)$ satisfies the Pohozaev identity. By a direct computation, we could prove that all the pairs of even integer solution of (\ref{pi-3}) is exactly $\Gamma(1,1)$. This proves Theorem \ref{thm-main} if (\ref{toda-b2g2}) has no singularities.
\medskip

\noindent {\em Proof of Corollary 1.5.} The first part is already proved. For the last part, we want to prove $|S_1|=1$. We observe: for any $(\sigma_1,\sigma_2)\in\Gamma(1,1)$, if $(\sigma_1,\sigma_2)$ is not special then $(\sigma_1,\sigma_2)$ can not be written as a sum of $(\sigma_1^{(1)},\sigma_2^{(1)})$ and $(\sigma_1^{(2)},\sigma_2^{(2)})$, where $(\sigma_1^{(i)},\sigma_2^{(i)})\in\Gamma(1,1).$ Now if $|S_1|\geq2$, then $(\sigma_1(\tau_s^k),\sigma_2(\tau_s^k))$ can be written as a sum of $(\sigma_1^{(1)},\sigma_2^{(1)})$ and $(\sigma_1^{(2)},\sigma_2^{(2)})$, where $(\sigma_1^{(i)},\sigma_2^{(i)})\in\Gamma(1,1).$ But if $(\sigma_1(\tau_s^k),\sigma_2(\tau_s^k))$ is not special, then it can not be written in this way. \hfill  $\square$
\medskip

If $0$ is a singularity of (\ref{toda-b2g2}) then $\Sigma_k$ could be written as a disjoint union of $\{0\}$ and $S_j$ ($j=1,..,m$). Here each $S_j$ is collected by the process described above and is maximal in the following sense:
\medskip

\noindent(i). $0\not \in S$, $|S|\ge 2$ and for any two distinct points $x_i^k,x_j^k$ in $S$ we have $$\mbox{dist}(x_i^k,x_j^k)\ll\tau^k(S),$$
\indent ~where $\tau^k(S)=\mbox{dist}(S,\Sigma_k\setminus S)$.

\noindent(ii). For any $0\neq x_i^k\in \Sigma_k\setminus S$,
$$\mbox{dist}(x_i^k,0)\le C \mbox{dist}(x_i^k,S) $$
for some constant $C$.
\medskip

For $S_j$ we define
$$\tau_{S_j}^k=\frac 12\mbox{dist}(S_j,\Sigma_k\setminus S_j). $$
Then the process described above proves the main result of this section:

\begin{prop}
\label{loc-mass-2}
Let $S_j$ ($j=1,..,m$) be described as above, then
(\ref{52-e3})-(\ref{52-e4}) holds where $B(x_1^k,\tau_S^k)$ is replaced by $B(x_i^k,\tau_{S_j}^k)$ and $x_i^k$ is any element in $S_j$.
\end{prop}
\medskip
%

\section{Proof of Theorem \ref{one-concen}, Theorem \ref{thm-main}, Theorem 1.6 and Theorem 1.7}

In Proposition \ref{loc-mass-2}, we write $\Sigma_k=\{0\}\cup S_1\cup\cdots\cup S_N$. From the construction, the ratio $\frac{|x^k|}{|\tilde{x}^k|}\leq C$ for any $x^k,{\tilde x}^k\in S_j$. Let
$$\|S_j\|=\min_{x^k\in S_j}|x^k|$$
and arrange $S_j$ by
$$\|S_1\|\leq\|S_2\|\leq\cdots\leq\|S_N\|.$$
Assume $l$ is the largest number such that $\|S_l\|\leq C\|S_1\|$. Then $\|S_l\|\ll\|S_{l+1}\|$.
\medskip

We recall the local mass contributed by $x_j^k\in S_j$ is
$$(\hat{\sigma}_1(B(x_{j}^{k},\tau^{k}_{j})),\hat{\sigma}_2(B(x_{j}^{k},\tau^{k}_{j})))=(m_{1,j},m_{2,j}),\quad
\mbox{where}~m_{1j},m_{2j}\in2\mathbb{N}\cup\{0\}.$$
Let
$$r_1^k=\frac12\|S_1\|.$$
By Proposition \ref{local-mass-type-1}, we have
\begin{align*}
u_i^k(x)+2\log|x|\leq C~\mathrm{for}~0<|x|\leq r_1^k, i=1,2.
\end{align*}
\medskip

\noindent{\em Proof of Theorem \ref{thm-main}.} Let
\begin{align*}
\tilde{u}_i^k(x)=u_i^k(x)+2\alpha_i\log|x|,~i=1,2.
\end{align*}
Then equation (\ref{toda-b2g2}) becomes
\begin{align*}
\Delta\tilde{u}_i^k(x)+\sum_{j=1}^2k_{ij}|x|^{2\alpha_j}h_j^k(x)e^{\tilde{u}_j^k(x)}=0,~|x|\leq r_1^k,~i=1,2.
\end{align*}
Let
\begin{equation}
\label{6.1}
-2\log\delta_k=\max_{i\in I}\max_{x\in\bar B(0,\tau_1^k)}\frac{\tilde{u}_i^k}{1+\alpha_i},
\end{equation}
then $\tilde{v}_i^k(y)$ defined as
\begin{equation}
\label{6.2}
\tilde{v}_i^k(y)=\tilde{u}_i^k(\delta_ky)+2(1+\alpha_i)\log\delta_k,~|y|\leq r_1^k/\delta_k,~i=1,2,
\end{equation}
satisfies
\begin{equation}
\label{6.3}
\Delta\tilde{v}_i^k(y)+\sum_{j\in I}k_{ij}|y|^{2\alpha_j}h_j^k(\delta_ky)e^{\tilde{v}_j^k(y)}=0,~|y|\leq r_1^k/\delta_k,~i=1,2.
\end{equation}
We have either
$$(a)~\lim_{k\to\infty} r_1^k/\delta_k=\infty\quad\mathrm{or}\quad(b)~r_1^k/\delta_k\leq C.$$
For (a). Our purpose is to prove a similar result as Proposition \ref{local-mass-type-1}:
\begin{enumerate}
  \item [(1).] {\em At most one component of $u^k$ has slow decay on $\partial B(0,r_1^k).$}

  \noindent As in section 5, we define
  \begin{equation*}
  \hat\sigma_{i,1}=\left\{\begin{array}{ll}
  \lim_{k\to+\infty}\sigma_i^k(B(0,r_1^k)) &\mbox{if}~u_i^k~\mbox{has fast decay on}~\partial B(0,r_1^k),\\
  \lim_{r\to0}\lim_{k\to+\infty}\sigma_i^k(B(0,rr_1^k)) &\mbox{if}~u_i^k~\mbox{has slow decay on}~\partial B(0,r_1^k),
  \end{array}\right.
  \end{equation*}
  \item [(2).] $(\hat\sigma_{1,1},\hat\sigma_{2,1})$ satisfies the Pohozaev identity (1.8), and
  \item [(3).] $\hat\sigma_{i,1}=2\sum_{i=1}^2n_i\mu_i+2n_3,~n_i\in\mathbb{Z}.$
\end{enumerate}
We carry out the proof in the discussion of the following two cases.
\medskip

\noindent Case 1. If both $\tilde{v}_i^k(y)$ converge in any compact set of $\mathbb{R}^2$, $(\sigma_1,\sigma_2)$ can be obtained by the classification theorem in \cite{lin-wei-ye}:
\begin{equation*}
(\sigma_1,\sigma_2)=\left\{\begin{array}{ll}
(2\mu_1+2\mu_2,2\mu_1+2\mu_2)~&\mathrm{for}~A_2,\\
(4\mu_1+2\mu_2,4\mu_1+4\mu_2)~&\mathrm{for}~B_2,\\
(8\mu_1+4\mu_2,12\mu_1+8\mu_2)~&\mathrm{for}~G_2.
\end{array}\right.
\end{equation*}
By Lemma \ref{le4.1}, both $u_i^k$ have fast decay on $\partial B(0,r_1^{k})$. So this proves (1)-(3) in this case.
\medskip

\noindent Case 2. Only one $\tilde{v}_i^k$ converges to $v_i(y)$ and the other tends to $-\infty$ uniformly in any compact set. Then it is easy to see that there is $l_k\ll r_1^k$ such that both $u_i^k$ have fast decay on $\partial B(0,l_k)$ and
$$(\sigma_1(B(0,l_k)),\sigma_2(B(0,l_k)))=(2\mu_1,0)~\mathrm{or}~(\sigma_1(B(0,l_k)),\sigma_2(B(0,l_k))=(0,2\mu_2)).$$
So this is the same situation as in the starting point for Proposition \ref{local-mass-type-1}. Then the same argument of Proposition \ref{local-mass-type-1} leads to the conclusion (1)-(3).

The pair $(\hat\sigma_{1,1},\hat\sigma_{2,1})$ can be calculated by the same method in Proposition 5.1. Then $(\hat\sigma_{1,1},\hat\sigma_{2,1})\in\Gamma(\mu_1,\mu_2)$, which is given in section 2.

To continue our discussion for $r\in[r_1^k,r_2^k]$, where we denote $\frac12\|S_{l+1}\|$ by $r_2^k$. We separate our discussion into two cases also.
\medskip

\noindent Case 1. One component has slow decay on $\partial B(0,r_1^k)$, say $u_1^k$. Then we scale
$$v_i^k(y)=u_i^k(r_1^ky)+2\log r_1^k.$$
By our assumption $v_1^k(y)$ converges to $v_1(y)$ and $v_2^k(y)\to-\infty$ in any compact set. Let $x_j^k\in S_j$ and $y_j^k=(r_1^k)^{-1}x_j^k\to p_j$. Then $v_1(y)$ satisfies
\begin{equation}
\label{6.4}
\Delta v_1+2e^{v_1}=4\pi\tilde{\alpha}_1\delta_0+4\pi\sum_{j=1}^m\tilde{n}_{1,j}\delta_{p_j},
\end{equation}
where
\begin{equation}
\label{6.5}
\tilde{n}_{1,j}=-\frac12\sum_{i=1}^2k_{1i}m_{i,j}~\mathrm{for~some}~m_{ij}\in\mathbb{Z}~\mathrm{and}
~\tilde\alpha_1=
\alpha_1-\hat{\sigma}_{1,1}+\frac12\hat\sigma_{2,1}.
\end{equation}
The finiteness of $\int_{\mathbb{R}^2}e^{v_1}$ implies that
$$\tilde\alpha_1>-1~\mathrm{and}~\tilde{n}_{1,j}\geq0.$$
By Theorem 2.2, we have
\begin{equation}
\label{6.6}
\frac{1}{2\pi}\int_{\mathbb{R}^2}e^{v_1}\mbox{d}y=2(\tilde\alpha_1+1)+2k_1~\mathrm{or}~\frac{1}{2\pi}\int_{\mathbb{R}^2}e^{v_1}\mbox{d}y=2k_2,
~\mathrm{where}~k_1,k_2\in\mathbb{Z}.
\end{equation}
As before, we can choose $l_k$, $r_1^k\ll l_k\ll r_2^k$ such that both $u_i^k$ have fast decay on $\partial B(0,l_k)$. Then the new pair $(\hat{\sigma}_{1,2},\hat{\sigma}_{2,2})$, which defined by
$$\hat{\sigma}_{t,2}=\frac{1}{2\pi}\lim_{k\to 0}\int_{B(0,l_k)}h_t^ke^{v_t^k},~t=1,2,$$
becomes
\begin{equation}
\label{6.7}
(\hat\sigma_{1,2},\hat\sigma_{2,2})=(\hat\sigma_{1,1}+\frac{1}{2\pi}\int_{\mathbb{R}^2}e^{v_1}+\sum_{j=1}^mm_{1,j},~\hat\sigma_{2,1}+\sum_{j=1}^mm_{2,j})
\end{equation}
for $m_{1j},m_{2j}\in2\mathbb{N}\cup\{0\}$. Using (\ref{6.6}), we get
\begin{equation}
\label{6.8}
\hat{\sigma}_{1,2}=\left\{\begin{array}{ll}
\hat{\sigma}_{1,1}+2k_2+\sum_{j=1}^{m}m_{1,j},~&\mathrm{if}~\frac{1}{2\pi}\int_{\mathbb{R}^2}e^{v_1}\mbox{d}y=2k_2,\\
2\mu_1+\hat\sigma_{2,1}-\hat\sigma_{1,1}+2k_1+\sum_{j=1}^mm_{1,j},~&\mathrm{if}~\frac{1}{2\pi}\int_{\mathbb{R}^2}e^{v_1}\mbox{d}y=2(\tilde\alpha_1+1)+2k_1.
\end{array}\right.
\end{equation}
We note that if $(\hat\sigma_{1,1},\hat\sigma_{2,1})\in\Gamma(\mu_1,\mu_2)$ and $2\mu_1+\hat\sigma_{2,1}-\hat\sigma_{1,1}>0$, then $$(2\mu_1+\hat\sigma_{2,1}-\hat\sigma_{1,1},\hat\sigma_{2,1})\in\Gamma(\mu_1,\mu_2).$$
Let $(\sigma_1^*,\sigma_2^*)=(2\mu_1+\hat\sigma_{2,1}-\hat\sigma_{1,1},~\hat{\sigma}_{2,1})$, we can write
\begin{equation}
\label{6.9}
(\hat\sigma_{1,2},\hat\sigma_{2,2})=(\sigma_1^*+m_1,~\sigma_2^*+m_2),
\end{equation}
with $(\sigma_1^*,\sigma_2^*)\in\Gamma(\mu_1,\mu_2)$ and $m_1,m_2\in2\mathbb{Z}.$
\medskip

\noindent Case 2. If both $u_i^k$ have fast decay on $\partial B(0,r_1^k),$ then they have fast decay on $\partial B(0,cr_1^k)$, where we choose $c$ bounded such that $\bigcup_{j=1}^mS_j\subset B(0,\frac{c}{2}r_1^k)$. Then the new pair $(\hat{\sigma}_{1,2},\hat{\sigma}_{2,2})$ becomes
\begin{equation*}
(\hat\sigma_{1,2},\hat\sigma_{2,2})=(\hat\sigma_{1,1}+\sum_{j=1}^mm_{1,j},~\hat\sigma_{2,1}+\sum_{j=1}^mm_{2,j})~\mbox{for}~m_{1,j},m_{2,j}\in2\mathbb{Z}.
\end{equation*}
Hence, in this case we can also write
\begin{equation}
\label{6.10}
(\hat\sigma_{1,2},\hat\sigma_{2,2})=(\sigma_1^*+m_1,~\sigma_2^*+m_2)
\end{equation}
with $(\sigma_1^*,\sigma_2^*)=(\hat\sigma_{1,1},\hat\sigma_{2,1})\in\Gamma(\mu_1,\mu_2)$ and $m_1,m_2\in2\mathbb{Z}.$ Denote $cr_1^k=l_k$. Then we can continue our process starting from $l_k$. After finite steps, we could prove that at most one component $u_1^k$ has slow decay on $\partial B(0,1)$ and their local masses have the expression in (3).

For case (b), i.e. $r^k_1/\delta_k\leq C$. Using $\tilde{v}_i^k\leq0$ we have $|y|^{2\alpha_j}h_j^k(\delta_ky)e^{\tilde{v}_j^k}\leq C$ on $B(0,r_1^k/\delta_k).$ Combined with the fact that $\tilde{v}_i^k$ has bounded oscillation on $\partial B(0,\tau_k/\delta_k)$ and $\tilde{v}_i^k\leq 0$ we get
\begin{equation*}
\tilde{v}_i^k(x)=\bar{\tilde{v}}_i^k(\partial B(0,r_1^k/\delta_k))+O(1)~\mathrm{for~all}~x\in B(r_1^k/\delta_k),
\end{equation*}
where $\bar{\tilde{v}}_i^k(\partial B(0,r_1^k/\delta_k))$ stands for the average of $\tilde{v}_i^k$ on $\partial B(0,r_1^k/\delta_k)$. Direct computation shows that
\begin{equation*}
\int_{B(0,r_1^k)}h_i^ke^{u_i^k}\mathrm{d}x=\int_{B(0,r_1^k/\delta_k)}h_i^k(\delta_k y)e^{\tilde{v}_i^k(y)}\mathrm{d}y=O(1)e^{\bar{\tilde{v}}_i^k(\partial B(0,r_1^k/\delta_k))}.
\end{equation*}
So $\int_{B(0,r_1^k)}h_i^ke^{u_i^k}\mathrm{d}x=o(1)$ if $\bar{\tilde{v}}_i^k(\partial B(0,r_1^k/\delta_k))\rightarrow-\infty$. On the other hand, we note that $\bar{\tilde{v}}_i^k(\partial B(0,r_1^k/\delta_k))\rightarrow-\infty$ is equivalent to $u_i^k$ having fast decay on $\partial B(0,r_1^k)$. As a consequence, we have $\hat\sigma_{i,1}=0$ if $u_i^k$ has fast decay on $\partial B(0,r_1^k)$. So if both two components have fast decay on $\partial B(0,r_1^k)$ we have $(\hat\sigma_{1,1},\hat\sigma_{1,2})=(0,0)$.

If some component of $u^k$ has slow decay on $\partial B(0,r_1^k)$, say $u_2^k$, then we choose $\tilde{l}_k$, $r_1^k\ll\tilde{l}_k\ll r_2^k$ such that
$$\sigma_2(B(0,\tilde{l}_k))=\sigma_2(B(0,r_1^k))=0$$
and both $u_i^k$ have fast decay on $\partial B(0,\tilde{l}_k)$. Then $(\sigma_1(B(0,\tilde{l}_k)),\sigma_2(B(0,\tilde{l}_k)))$ satisfies (\ref{6.9}) with $(\sigma_1^*,\sigma_2^*)=(\hat\sigma_{1,1},\hat\sigma_{2,1})=(0,0)$, which implies
$$\sigma_1(B(0,\tilde{l}_k))=0~\mathrm{or}~\sigma_1(B(0,\tilde{l}_k))=2\mu_1.$$
Hence in both cases, we could choose $\tilde{l}_k\ll r_2^k$ such that (1)-(3) holds on $\partial B(0,\tilde{l}_k).$ Afterwards, we continue our discussion as the case (a). Then Theorem 1.3 is proved completely. \hfill $\square$
\medskip

Next, we shall prove Theorem 1.6, that is $\Sigma_{k}=\{0\}$ by way of contradiction. Suppose $\Sigma_k$ has points other than $0.$ Using the notations from the beginning of this section, we have
$$\Sigma_k=\{0\}\cup S_1\cup\cdots\cup S_N.$$
Now suppose $r_1^k/\delta_k\rightarrow\infty$ as $k\to\infty$. Let $(\hat{\sigma}_{1,2},\hat\sigma_{2,2})$ be the local masses defined by (\ref{6.7}) for one of the component $u_i^k$ having slow decay on $\partial B(0,r_1^k)$ or by (\ref{6.8}) for both two components having fast decay on $\partial B(0,r_1^k)$. Then we recall the following result
\begin{itemize}
  \item[(i).] $\hat\sigma_{i,2}=\sigma^*_{i}+m_i,$ where $(\sigma_1^*,\sigma_2^*)\in\Gamma(\mu_1,\mu_2)$ and $m_i,i=1,2$ are even integers.
  \item[(ii).] Both pairs $(\sigma^*_{1},\sigma^*_{2})$ and $(\hat{\sigma}_{1,2},\hat\sigma_{2,2})$ satisfy the Pohozaev identity.
\end{itemize}

Based on the description above, we are able to prove Theorem 1.6.
\medskip

\noindent {\em Proof of Theorem 1.6.}
From the above discussion, we have
$$(\hat\sigma_{1,2},\hat\sigma_{2,2})=(\sigma^*_{1}+m_1,~\sigma^*_{2}+m_2).$$
We note that the conclusion of Theorem 1.5 is equivalent to show $m_i=0,i=1,2.$ In the following, we shall prove
$$m_i=0,i=1,2.$$
From the above discussion, we have both $(\hat\sigma_{1,2},\hat\sigma_{2,2})$ and $(\sigma_1^*,\sigma_2^*)$ satisfy the Pohozaev identity
\begin{equation}
\label{t7.1}
k_{21}\sigma_1^2+k_{12}k_{21}\sigma_1\sigma_2+k_{12}\sigma_2^2=2k_{21}\mu_1\sigma_1+2k_{12}\mu_2\sigma_2.
\end{equation}
We can write them as
\begin{equation}
\label{t7.2}
k_{21}(\sigma^*_{1})^2+k_{12}k_{21}\sigma^*_{1}\sigma^*_{2}+k_{12}(\sigma^*_{2})^2=2k_{21}\mu_1\sigma^*_{1}+2k_{12}\mu_2\sigma^*_{2},
\end{equation}
and
\begin{equation}
\label{t7.3}
\begin{aligned}
&k_{21}(\sigma^*_{1}+m_1)^2+k_{12}k_{21}(\sigma^*_{1}+m_1)(\sigma^*_{2}+m_2)+k_{12}(\sigma^*_{2}+m_2)^2\\
&=2k_{21}\mu_1(\sigma^*_{1}+m_1)+2k_{12}\mu_2(\sigma^*_{2}+m_2).
\end{aligned}
\end{equation}
We use \eqref{t7.3} and \eqref{t7.2} to get
\begin{equation}
\label{t7.4}
\begin{aligned}
&2k_{21}m_1\hat\sigma^*_{1}+k_{12}k_{21}m_2\hat\sigma^*_{1}+k_{12}k_{21}m_1\sigma^*_{2}+2k_{12}m_2\sigma^*_{2}\\
&=2k_{21}m_1\mu_1+2k_{12}m_2\mu_2-(k_{21}m_1^2+k_{12}k_{21}m_1m_2+k_{12}m_2^2).
\end{aligned}
\end{equation}
Since $(\sigma_1^*,\sigma_2^*)\in\Gamma(\mu_1,\mu_2)$, we set
$$\sigma^*_{1}=l_{1,1}\mu_1+l_{1,2}\mu_2,\quad\sigma^*_{2}=l_{2,1}\mu_1+l_{2,2}\mu_2.$$
Then we can rewrite \eqref{t7.4} as
\begin{equation}
\label{t7.5}
\begin{aligned}
&\left(2k_{21}l_{1,1}m_1+k_{12}k_{21}l_{2,1}m_1-2k_{21}m_1+2k_{12}l_{2,1}m_2+k_{12}k_{21}l_{1,1}m_2\right)\mu_1\\
&+\left(2k_{21}l_{1,2}m_1+k_{12}k_{21}l_{2,2}m_1+2k_{12}l_{2,2}m_2+k_{12}k_{21}l_{1,2}m_2-2k_{12}m_2\right)\mu_2\\
&+(k_{21}m_1^2+k_{12}k_{21}m_1m_2+k_{12}m_2^2)=0.
\end{aligned}
\end{equation}
Since $\mu_1,\mu_2$ and $1$ are linearly independent, we have the coefficients of $\mu_1$ and $\mu_2$ must vanish. Equivalently we have
\begin{equation}
\label{t7.6}
\left(\begin{matrix}
2k_{21}l_{1,1}+k_{12}k_{21}l_{2,1}-2k_{21} & 2k_{12}l_{2,1}+k_{12}k_{21}l_{1,1}\\
2k_{21}l_{1,2}+k_{12}k_{21}l_{2,2} & 2k_{12}l_{2,2}+k_{12}k_{21}l_{1,2}-2k_{12}
\end{matrix}\right)
\left(\begin{matrix}
m_1\\
m_2
\end{matrix}\right)=0
\end{equation}
Denote $M_{K}$ as
\begin{align*}
M_K=\left(\begin{matrix}
2k_{21}l_{1,1}+k_{12}k_{21}l_{2,1}-2k_{21} & 2k_{12}l_{2,1}+k_{12}k_{21}l_{1,1}\\
2k_{21}l_{1,2}+k_{12}k_{21}l_{2,2} & 2k_{12}l_{2,2}+k_{12}k_{21}l_{1,2}-2k_{12}
\end{matrix}\right).
\end{align*}
{\bf Claim}: $M_k$ is non-singular. We shall divide our proof into the following three cases.
\medskip

\noindent Case 1. $\mathbf{K}=A_2$. Then we can write (\ref{t7.6}) as
\begin{equation}
\label{t7.7}
\left(\begin{matrix}
2l_{1,1}-l_{2,1}-2 & 2l_{2,1}-l_{1,1}\\
2l_{1,2}-l_{2,2} & 2l_{2,2}-l_{1,2}-2
\end{matrix}\right)
\left(\begin{matrix}
m_1\\
m_2
\end{matrix}\right)=0.
\end{equation}
We note that
\begin{align*}
(l_{1,1},l_{1,2},l_{2,1},l_{2,2})\in\{(2,0,0,0),(0,0,0,2),(2,2,0,2),(2,0,2,2),(2,2,2,2)\}.
\end{align*}
Then it is easy to see that $M_K$ is non-singular when $(l_{1,1},l_{1,2},l_{2,1},l_{2,2})$ belongs the above set.
\medskip

\noindent Case 2. $\mathbf{K}=B_2$. Then we can write (\ref{t7.6}) as
\begin{equation}
\label{t7.8}
\left(\begin{matrix}
2l_{1,1}-l_{2,1}-2 & l_{2,1}-l_{1,1}\\
2l_{1,2}-l_{2,2} & l_{2,2}-l_{1,2}-1
\end{matrix}\right)
\left(\begin{matrix}
m_1\\
m_2
\end{matrix}\right)=0.
\end{equation}
We note that
\begin{align*}
(l_{1,1},l_{1,2},l_{2,1},l_{2,2})\in\{&(2,0,0,0),(2,0,4,2),(4,2,4,2),(0,0,0,2),\\
&(2,2,0,2),(2,2,4,4),(4,2,4,4)\}
\end{align*}
From the above set, we can see that $4|(l_{2,1}l_{1,1})(2l_{1,2}-l_{2,2})$. As a result, if the determinant of $M_K$ is $0$, we have to make $4|(2l_{1,1}-l_{2,1}-2)$, which forces $l_{2,1}\equiv2(mod4)$. However, this is impossible. Thus $M_k$ is non-singular in this case.
\medskip

\noindent Case 3. $\mathbf{K}=G_2$. Then we can write (\ref{t7.6}) as
\begin{equation}
\label{t7.9}
\left(\begin{matrix}
6l_{1,1}-3l_{2,1}-6 & 2l_{2,1}-3l_{1,1}\\
6l_{1,2}-3l_{2,2} & 2l_{2,2}-3l_{1,2}-2
\end{matrix}\right)
\left(\begin{matrix}
m_1\\
m_2
\end{matrix}\right)=0.
\end{equation}
We note that
\begin{align*}
(l_{1,1},l_{1,2},l_{2,1},l_{2,2})\in\{&(2,0,0,0),(2,0,6,2),(6,2,6,2),(6,2,12,6),\\
&(8,4,12,6),(8,4,12,8),(0,0,0,2),(2,2,0,2),\\
&(2,2,6,6),(6,4,6,6),(6,4,12,8)\}.
\end{align*}
From the above list, we note $3|l_{2,1}$, then we get $9|(2l_{2,1}-3l_{1,1})(6l_{1,2}-3l_{2,2})$. On the other hand, we note
$$l_{1,1}\equiv0,2(mod3)~\mathrm{and}~l_{2,2}\equiv0,2(mod3),$$
this implies $(6l_{1,1}-3l_{2,1}-6)(2l_{2,2}-3l_{1,2}-2)$ is not multiple of $9$, therefore we have the determinant of $M_K$ is not zero. Thus we have $M_k$ is non-singular when $\mathbf{K}=G_2.$

From the above discussion, we have $(m_1,m_2)=(0,0)$. Therefore, Theorem 1.6 is proved completely. \hfill $\square$
\medskip

At the end, we give the proof of Theorem \ref{one-concen} and Theorem 1.7.
\medskip

\noindent {\em Proof of Theorem \ref{one-concen} and Theorem 1.7.}
Suppose there exists a sequence of blow up solutions $(u_1^k,u_2^k)$ of (\ref{e-gen}) with $(\rho_1,\rho_2)=(\rho_1^k,\rho_2^k)$. At first, we prove Theorem \ref{one-concen}. From the previous discussion of this section, we get that at least one component (say $u_1^k$) of $u^k$ has fast decay on a small ball $B$ near each blow up point $q$, which means $u_1^k(x)\rightarrow-\infty$ if $x\not\in S$ and $x$ is not a blow up point. Hence Theorem \ref{one-concen} holds.

Because the mass distribution of $u_1^k$ is concentrate as $k\to+\infty,$ we get that $\lim_{k\to+\infty}\rho_1^k$ is equal to the sum of the local mass $\sigma_1$ at a blow up point $q$, which implies $\rho_1\in\Gamma_1,$ a contradiction to the assumption. Thus, we finish the proof of the Theorem 1.7.
\hfill $\square$


\vspace{1cm}





\begin{thebibliography}{99}
\bibitem{ao-1} W. Ao, C. S. Lin, J. C. Wei,  On Non-topological Solutions of the $A_2$ and $B_2$ Chern-Simons System, {\em  Mem. Amer. Math. Soc. 239 (2016), no. 1132, v+88 pp. ISBN: 978-1-4704-1543-3; 978-1-4704-2747-4}.

\bibitem{bart2} D. Bartolucci, G. Tarantello, The Liouville equation with singular data: a concentration-compactness principle via a local representation
formula. {\em J. Differential Equations} 185 (2002), no. 1, 161-180.


\bibitem{malchiodi-b} L. Battaglia, A. Malchiodi, A Moser-Trudinger Inequality for the singular Toda system, {\em  Bull. Inst. Math. Acad. Sin. (N.S.)} 9 (2014), no. 1, 1-23.

\bibitem{bennet} W. H. Bennet, Magnetically self-focusing streams, {\em Phys. Rev.} 45 (1934), 890-897.



\bibitem{bm} H. Brezis, F. Merle, Uniform estimates and blow-up behavior for solutions of $-\Delta u =V(x) e^{u} $ in two dimensions. {\em Comm. Partial Differential Equations} 16 (1991), no 8-9, 1223-1253.


\bibitem{chai} C. L. Chai, C. S. Lin, C. L. Wang, Mean field equations, hyperelliptic curves, and modular forms: I. {\em Camb. J. Math.} 3 (2015), no. 1-2, 127-274.

\bibitem{chenlin1} C. C. Chen, C. S. Lin, Sharp estimates for solutions of multi-bubbles in compact Riemann surfaces.  {\em Comm. Pure Appl. Math. }  55  (2002),  no. 6, 728-771.

\bibitem{chenlin2} C. C. Chen, C. S. Lin, Topological degree for a mean field equation on Riemann surfaces. {\em Comm. Pure Appl. Math. } 56 (2003), no. 12, 1667-1727.


\bibitem{CKL} Z. J. Chen, T. Y. Kuo, C. S. Lin, Green function, Painlev\`e equation, and Eisenstein series of weight one. To appear in {\em J. Differential Geometry.}

\bibitem{lin-cheng} K. S.  Cheng, C. S. Lin, On the asymptotic behavior of solutions of the conformal Gaussian curvature equations in $\mathbb{R}^2$. {\em Math. Ann.} 308 (1997), no. 1, 119-139.

\bibitem{chern} S. S. Chern, J. G. Wolfson, Harmonic maps of the two-sphere into a complex Grassmann manifold. II. {\em Ann. Math.} 125(2), 301--335 (1987).

\bibitem{doliwa} A. Doliwa, Holomorphic curves and Toda systems. {\em Lett. Math. Phys.} 39(1), 21-32 (1997).




\bibitem{erem-3} A. Eremenko, A. Gabrielov, V. Tarasov, Metrics with conic singularities and spherical polygons, {\em Illinois J. Math,} vol 58, no3, 2014, 739-755.

\bibitem{ganoulis} N. Ganoulis, P. Goddard, D. Olive, Self-dual monopoles and Toda molecules. {\em  Nucl. Phys. B } 205, 601-636 (1982).


\bibitem{jostlinwang} J. Jost, C. S. Lin, G. F. Wang, Analytic aspects of the Toda system II: bubbling behavior and existence of solutions, {\em Comm. Pure Appl. Math.} 59 (2006), no. 4, 526-558.

\bibitem{lee2} K. Lee, Selfdual nonabelian Chern-Simons solitons {\em Phys. Lett.}, 66 (1991), 552-555.

\bibitem{lee} Y. Lee, C. S. Lin, J. C. Wei, W. Yang, Degree counting and Shadow system for Toda system of rank two: one bubbling. Preprint.

\bibitem{licmp} Y. Y. Li, Harnack type inequality: the method of moving planes. {\em Comm. Math. Phys.} 200 (1999), no. 2, 421-444.

\bibitem{li-shafrir} Y. Y. Li, I. Shafrir, Blow-up analysis for solutions of $-\Delta u=V(x)e^u$ in dimension two. {\em Indiana Univ. Math. J.} 43(1994), 1255-1270.

\bibitem{lin-tarantello} C. S. Lin, G. Tarantello, When Blowup does not imply concentration: a detour from Brezis-Merle's result. {\em C. R. Math. Acad. Sci. Paris} 354 (2016), no. 5, 493-498.


\bibitem{lin-wei-ye} C.S. Lin, J. C. Wei, D. Ye, Classifcation and nondegeneracy of $SU(n+1)$ Toda system, {\em Invent. Math. } 190(2012), no.1, 169-207.

\bibitem{lwz-apde} C. S. Lin, J. C. Wei, L. Zhang, Classification of blowup limits for SU(3) singular Toda systems, {\em   Anal. PDE} 8 (2015), no. 4, 807-837.

\bibitem{lwz-jems} C. S. Lin, J. C. Wei, L. Zhang, Local profile of fully bubbling solutions to SU(n+1) Toda systems, {\em  J. Eur. Math. Soc.} 18 (2016), no. 8, 1707-1728.

\bibitem{lin-yan1} C. S. Lin, S. S. Yan, Existence of bubbling solutions for Chern-Simons model on a torus. {\em Arch. Ration. Mech. Anal.} 207 (2013), no. 2, 353-392.

\bibitem{linzhang1} C. S. Lin, L. Zhang, Profile of bubbling solutions to a Liouville system.  {\em Ann. Inst. H. Poincare Anal. Non Lineaire} 27 (2010), no. 1, 117-143.

\bibitem{lin-zhang-jfa} C. S. Lin, L. Zhang, On Liouville systems at critical parameters, Part I: One bubble. {\em J. Funct. Anal. } 264 (2013), no 11, 2584-2636.

\bibitem{lin-zhang-b2g2} C. S. Lin, L. Zhang, Energy concentration for Singular Toda systems with $B_2$ and $G_2$ types of Cartan matrices. To appear in International Math. Research Notices.

\bibitem{MR} A. Malchiodi, D. Ruiz, a variational analysis of the Toda system on compact surfaces, {\em Comm. Pure Appl. Math.}  66 (2013), no. 3, 332-371.

\bibitem{mansfield} P. Mansfield, Solutions of Toda systems. {\em Nucl. Phys. B} 208, 277-300 (1982).

\bibitem{wei} M. Musso, A. Pistoia, J. C. Wei,  New blow-up phenomena for $SU(n+1)$ Toda system, {\em J. Differential Equations} 260 (2016), no. 7, 6232-6266.

\bibitem{nolasco2} M. Nolasco, G. Tarantello, Vortex condensates for the $SU(3)$ Chern-Simons theory. {\em Commun. Math. Phys. } 213(3), 599-639 (2000).

\bibitem{prajapat} J. Prajapat, G. Tarantello, On a class of elliptic problems in $\mathbb R^2$: symmetry and uniqueness results. {\em Proc. Roy. Soc. Edinburgh Sect. A} 131 (2001), no. 4, 967-985.

\bibitem{tarantello-1} G. Tarantello, Analytical aspects of Liouville-type equations with singular sources. Handbook of differential equations: Stationary partial differential equations. Vol. I, 491-592, Handb. Differ. Equ., North-Holland, Amsterdam, 2004.

\bibitem{tr-1}M. Troyanov, Prescribing curvature on compact surfaces with conical singularities, {\em Trans. AMS}, 324 (1991) 793-821.

\bibitem{tr-2} M. Troyanov, Metrics of constant curvature on a sphere with two conical singularities, {\em Lect. Notes Math., 1410, Springer, NY}, 1989, 296-308.

\bibitem{w-w} E. T. Whittaker, G. N. Watson, A Course of Modern Analysis, Cambridge University Press, 1902.

\bibitem{yang1} Y. Yang, The relativistic non-abelian Chern-Simons equation. {\em Commun. Math. Phys. } 186 (1), 199-218 (1999).

\bibitem{yang2} Y. Yang, Solitons in Field Theory and Nonlinear Analysis. {\em Springer Monographs in Mathematics}. Springer, New York (2001)

\end{thebibliography}
\end{document}